\documentclass[11pt,reqno]{amsart}

\usepackage{amsmath, amsfonts, amsthm, amssymb}

\newtheorem{Theorem}{Theorem}[section]
\newtheorem{Lemma}{Lemma}[section]
\newtheorem{Proposition}{Proposition}[section]

\theoremstyle{definition}
\newtheorem{Definition}{Definition}[section]

\theoremstyle{remark}
\newtheorem{Remark}{Remark}[section]

\numberwithin{equation}{section}

\renewcommand{\r}{\rho}

\def\i{\varepsilon}
\renewcommand{\u}{{\bf u}}
\renewcommand{\H}{{\bf H}}
\newcommand{\R}{{\mathbb R}}
\newcommand{\Dv}{{\rm div}}

\newcommand{\tr}{{\rm tr}}

\def\f{\frac}

\def\D{\Delta }

\def\hf1{^\f{1}{1-\xi^2}}

\def\be{\begin{equation}}
\def\en{\end{equation}}
\def\bs{\begin{split}}
\def\es{\end{split}}

\newcommand{\F}{{\mathtt F}}

\begin{document}

\author{Xianpeng Hu}
\address{Courant Institute of Mathematical sciences, New York
University, New York, NY 10012, USA.}
\email{xianpeng@cims.nyu.edu}

\title[Compressible Magnetohydrodynamics]
{Global Existence for Two Dimensional Compressible
Magnetohydrodynamic Flows with Zero Magnetic Diffusivity}


\date{\today}
\thanks{Supported by NSF Grant DMS-1108647.}

\begin{abstract}
The existence of global-in-time classical solutions to the Cauchy
problem of compressible magnetohydrodynamic flows with zero
magnetic diffusivity is considered in two dimensions. The linear structure is a degenerated hyperbolic-parabolic system. The solution is constructed as
a small perturbation of a constant background in critical spaces. The deformation
gradient is introduced to decouple the subtle coupling
between the flow and the magnetic field. The $L^1$ dissipation for the velocity is obtained, and the $L^2$ dissipations for the density and the magnetic field are also achieved.
\end{abstract}

\maketitle

\section{Introduction}

This work is devoted to the solvability of magnetohydrodynamic flows (MHD) with zero magnetic diffusivity, and concentrates on the compressible model. The compressible model is more
complicated to analyze than the incompressible model in \cite{HW2} since, in addition to a degenerated parabolic-hyperbolic structure between the flow and the magnetic field, 
it involves another parabolic-hyperbolic structure between the flow and the density. Recall that the equations describing
two-dimensional compressible magnetohydrodynamic flows in the
barotropic case with zero magnetic diffusivity have the following
form (\cite{Ca,KL, KL1,LL}):
\begin{equation} \label{e1}
\begin{cases}
\partial_t\rho +\Dv(\r\u)=0, \\
\partial_t(\r\u)+\Dv\left(\r\u\otimes\u-{\bf B}\otimes {\bf B}\right)+\nabla
\left(P(\r)+\f{1}{2}|{\bf B}|^2\right)=\mu \Delta \u+(\lambda+\mu)\nabla\Dv\u, \\
\partial_t {\bf B}-\nabla\times(\u\times {\bf B})=0,\quad
\Dv {\bf B}=0,
\end{cases}
\end{equation}
where $\r$ denotes the density, $\u\in\R^2$ is
the velocity, ${\bf B}\in\R^2$ is the magnetic field, and $P(\r)$ is the pressure of the flow. The viscosity coefficients of
the flow are independent of the magnitude and
direction of the magnetic field and satisfy $$\mu>0\quad\textrm{and}\quad\mu+\lambda\ge 0$$ which ensures the operator $-\mu \Delta \u-(\lambda+\mu)\nabla\Dv\u$ is strongly elliptic and can be deduced directly
from the second law of thermodynamics. The symbol $\otimes$ denotes the
Kronecker tensor product. Generally speaking, the pressure $P(\r)$ is an increasing
and convex function of $\r$, but \textit{for simplicity of the presentation}, it will be assumed in this paper that $P(\r)=\f13\r^3$. Usually, we refer to the first equation
in \eqref{e1} as the continuity equation, and the second equation
as the momentum balance equation. It is well-known that the
electromagnetic fields are governed by the Maxwell's equations. In
magnetohydrodynamics, the displacement current  can be neglected
(\cite{KL,LL}). As a consequence, the last equation in \eqref{e1}
is called the induction equation, and the electric field can be
written in terms of the magnetic field ${\bf B}$ and the
velocity $\u$,
\begin{equation*}
\mathfrak{E}=-\u\times {\bf B}.
\end{equation*}
Although the electric field $\mathfrak{E}$ does not appear in
\eqref{e1}, it is indeed induced according to the above relation
by the moving conductive flow in the magnetic field.

The global wellposedness of \eqref{e1} is expected physically and was observed numerically, but the rigorous mathematical verification is widely open since the pioneering work of Hannes Alfv$\acute{e}$n, a Nobel laureate, in 1940s.
In this paper, we are interested in global classical solutions $(\r,\u, {\bf B})$ to \eqref{e1} which are small perturbations around an equilibrium $(1,0,h_0)$, where, up to a scaling and a rotation of Eulerian coordinates,
the constant vector $h_0$ is assumed to be $(1,0)^\top$ in two dimensional space $\R^2$ ($v^\top$ means the
transpose of $v$). More precisely, we define
$${\bf B}=h_0+\H=(1+\H_1, \H_2)^\top\quad\textrm{and}\quad \r=1+b,$$ and consider the global classical solutions of
\begin{equation} \label{e2}
\begin{cases}
\partial_t\rho +\Dv(\r\u)=0, \\
\partial_t(\r\u)+\Dv\left(\r\u\otimes\u-\H\otimes \H\right)-h_0\cdot\nabla\H+\nabla
\left(\f{1}{3}\r^3+\H_1+\f{1}{2}|\H|^2\right)\\
\qquad\qquad\qquad=\mu \Delta \u+(\lambda+\mu)\nabla\Dv\u, \\
\partial_t {\bf H}-\nabla\times(\u\times {\bf H})=\nabla\times(\u\times h_0),\quad
\Dv \H=0,
\end{cases}
\end{equation}
associated with the initial condition:
\begin{equation}\label{IC1}
(b, \u, \H)|_{t=0}=(b_0(x), \u_0(x), \H_0(x)), \quad x\in\R^2.
\end{equation}

The global wellposedness of \eqref{e1} is a widely open problem due to
the strong coupling between the fluid and the magnetic field. Indeed, the main difficulty in solving \eqref{e2}-\eqref{IC1} lies in the lack of the dissipation mechanisms of the density and the magnetic field, which is sharply 
different from the situation for the compressible Navier-Stokes equation. Recall that the dissipation of the density for compressible Navier-Stokes equations is the backbone of the wellposedness in \cite{RD}.   
While the dissipations for one partial derivative of the magnetic field $\partial_{x_1}\H$ and the total pressure $\nabla(b+\H_1)$ are relatively clear, 
the dissipation for the other derivative of the magnetic field is very complicated and subtle (see \cite{LP})
and it turns out that the dissipation in this direction is intrinsically
related to the flow. This inspires us to carefully analyze the linear structure of \eqref{e2}
\begin{subequations}\label{e3}
 \begin{align}
&\partial_tb+\Dv \u=0\\
&\partial_t\u-\mu\D\u-(\lambda+\mu)\nabla\Dv\u-\partial_{x_1}\H+\nabla(b+\H_1)=0\\
&\partial_t\H+h_0\Dv\u=\partial_{x_1}\u.
 \end{align}
\end{subequations}
Indeed, taking one more derivative gives
\begin{equation}\label{e3a}
\partial_{tt}\u-\mu\D\partial_t\u-(\lambda+\mu)\nabla\Dv\partial_t\u-Q(D)\u=0,
\end{equation}
where 
$$Q(D)\u=\nabla\Dv\u+(0,\D\u_2)$$ with
\begin{equation*}
Q(D)=\left(\begin{array}{ccc}\partial_{x_1}^2&\partial_{x_1}\partial_{x_2}\\
             \partial_{x_1}\partial_{x_2}&\D+\partial_{x_2}^2
            \end{array}\right)\quad\textrm{and the symbol}\quad Q(\xi)=\left(\begin{array}{ccc}-\xi^2_1&-\xi_1\xi_2\\
             -\xi_1\xi_2&-|\xi|^2-\xi_2^2
            \end{array}\right).
\end{equation*}
Observe that the equation \eqref{e3a} is a degenerated parabolic-hyperbolic system.
In particular whenever $\Dv\u=0$ the equation \eqref{e3a} becomes
$$\partial_{tt}\u-\mu\D\u-\partial_{x_1}^2\u=0,$$ which is the linear structure of the incompressible MHD (see \cite{HW2,XZ}). The differential structure of \eqref{e3a} makes it necessary to control the
competition between the parabolicity and the hyperbolicity.
To overcome this difficulty, it requires us to diagonalize the system and carefully analyze the linear structure. 

One of the main novelties of this paper is the dissipation mechanism of
the magnetic field $\H$ and the density $b$. In order to achieve this goal, we introduce the concept of the deformation gradient (see \cite{CD,HW, LLZ1, LLZ,LZ,QZ, ST, ST1}),
which is defined to be the gradient of the flow map with respect to the Lagrangian configuration. The key observation
of obtaining those dissipation mechanisms is that the magnetic field
is closely related to the inverse of the deformation
gradient $\F$ (see Proposition \ref{p11}). The relation between the
deformation gradient and the magnetic field can be interpreted as
a ``frozen'' law in MHD (cf. \cite{Ca, KL}). It turns out that some combination
between the deformation gradient and the magnetic field is
transported by the flow (see Proposition \ref{p11}). A direct
consequence of this ``frozen'' law and properties of the
deformation gradient is that one can decouple the relation between
the density and the magnetic field (see subsection 2.1 below for details). From this viewpoint, the deformation
gradient serves like a bridge to connect the flow and the magnetic
field. 

Whenever the deformation gradient is involved, its $L^\infty$ norm needs to be under control. This intrigues a difficulty since the deformation gradient is not explicitly expressed in \eqref{e1}. To overcome this difficulty,
one way is to control $L^\infty$ norm of $\nabla\u$, since the deformation gradient satisfies 
a transport equation from its definition. For Cauchy problems, the necessarity of $L^\infty$ norms for $\nabla\u$ motives us to work on Besov spaces of functions, since $\dot{B}_{2,1}^1(\R^2)\subset L^\infty.$ This
actually coincides with the so-called \textit{Critical Spaces} of system \eqref{e1} (see \cite{RD,HW,QZ} for instance).  
Motivated by the work \cite{RD}, it is natural to handle the low and high frequence with different regularities for the density and the magnetic field. 
This consideration naturally involves the so-called \textit{hybrid Besov space} $\tilde{B}^{s,t}$ (see Definition in Section 2).  
In summary, if one is interested in the global existence of \eqref{e1} with as low as possible regularity, it seems necessary to work in the framework of Besov functional space because:
\begin{itemize}
\item The system \eqref{e1} is scaling invariant in Besov spaces $(\r, \u,{\bf B})\in \dot{B}_{2,1}^{1}\times\dot{B}^0_{2,1}\times\dot{B}_{2,1}^{1}$ with the embedding $ \dot{B}^1_{2,1}\subset L^\infty$ in $\R^2$;
\item Due to the weak dissipation mechanism for the magnetic field in the direction parallel to the background, 
the $L^1$ integrability of $\|\nabla\u\|_{L^\infty}$ is necessary and which is generally obtained by using Besov spaces for Cauchy problems.
\end{itemize}

When the density is a constant, system \eqref{e1} takes the form
of incompressible MHD with zero magnetic diffusivity. For the
incompressible version of \eqref{e1}, authors in \cite{XZ, XZ1}
proposed a global existence near the equilibrium $(\u,\H)=(0,h_0)$
in Lagrangian coordinates in anisotropic Besov spaces. Recently authors in \cite{HW2} verified the global wellposedness in Eulerean coordinates in the framework of \textit{hybrid Besov spaces}. As
the viscosities $\mu$ and $\lambda$ vanish further, system
\eqref{e1} becomes two dimensional ideal incompressible MHD.
Authors in \cite{MZ, ZF} showed a local existence for the ideal
incompressible MHD (see also \cite{CG, Yu} for problems of
current-vortex sheets). As far as the criteria for blow-up in
finite time for system \eqref{e1} is concerned, we refer the
interested reader to \cite{LMZ}. For the MHD with partial
dissipation for the magnetic field, authors in \cite{CW} showed a
global existence in two dimension without the smallness
assumption. For one dimensional ideal compressible MHD, we refer
the interested reader to \cite{gw, gw2, HT, KS, sm, w1}. On the
other hand, as the magnetic diffusivity presents, global weak solutions
have been constructed in \cite{f3, HW1, MR}.

This paper is organized as follows. In Section 2, we will state
our main result and explain the strategy of proof, including the
dissipation mechanism for perturbations. In Section 3 \textit{hybrid Besov spaces} are introduced and basic properties of these spaces are analyzed. Section 4 and Section 5 are devoted
to uniform dissipation estimates for perturbations, while we will
finish the proof of our main result in Section 6. Throughout this
paper, $A_i$ means the $i-$th component (column respectively) of a
vector (matrix respectively) $A$. The notation $(\cdot|\cdot)$ stands for the standard inner product in Lebesgue space $L^2$.

%%%%%%%%%%%%%%%%%%%%%%%%%%%%%%%%%%%%%%%%%%%%%%%%%%%%%%%%%%%%%%%%%%%%%%%%%%%%%%%%%%%%%%%%%

\bigskip

\section{Main Results and Strategy of Proof}

In this section, we state our main result and explain the strategy
of proof. For simplicity of presentation, we focus on the case: $\mu=1$ and $\lambda=-1$. The general case can be handled similarly since $-\mu\D-(\lambda+\mu)\nabla\Dv\u$ is a strongly elliptic operator. 
First we take a look at the linearized
structure of the momentum equation in \eqref{e1}. Indeed, using
the continuity equation, the second equation in \eqref{e2} can be
rewritten as
\begin{equation}\label{11}
 \partial_t\u-h_0\cdot\nabla\H-\D\u+\nabla\left(b+\H_1\right)=-\u\cdot\nabla\u+\mathcal{L}
\end{equation}
with
\begin{equation*}
 \begin{split}
\mathcal{L}&=\left(\f{1}{\r}-1\right)\Big[h_0\cdot\nabla\H+\D\u-\nabla\H_1\Big]+b\nabla b+\f{1}{\r}\left(\H\cdot\nabla\H-\f12\nabla|\H|^2\right).
 \end{split}
\end{equation*}
Define
$$\Lambda^s=\mathcal{F}^{-1}(|\xi|^s\mathcal{F}(f)),$$
where $\mathcal{F}$ denotes the Fourier transformation, and
denote
$$d=\Lambda^{-1}\Dv\u\quad\textrm{and}\quad \omega=\Lambda^{-1}\textrm{curl}\u$$
with $\textrm{curl}\u=\partial_{x_2}\u_1-\partial_{x_1}\u_2$. 
Applying operators $\Lambda^{-1}\Dv$ and $\Lambda^{-1}\textrm{curl}$ to \eqref{11} respectively yields
\begin{equation}\label{12}
\partial_t d-\D d-\Lambda(b+\H_1)=-\Lambda^{-1}\Dv(\u\cdot\nabla\u-\mathcal{L})
\end{equation}
and
\begin{equation}\label{13}
\partial_t\omega-\D\omega-\Lambda\H_2=-\Lambda^{-1}\textrm{curl}(\u\cdot\nabla\u-\mathcal{L}).
\end{equation}
Here for \eqref{12} and \eqref{13}, we used $\Dv\H=0$ and
\begin{equation*}
 \begin{split}
\Lambda^{-1}\textrm{curl}(h_0\cdot\nabla\H)&=\Lambda^{-1}\left(\f{\partial^2 \H_1}{\partial x_1x_2}-\f{\partial^2\H_2}{\partial x_1^2}\right)\\
&=-\Lambda^{-1}\left(\f{\partial^2\H_2}{\partial x_1^2}+\f{\partial^2 \H_2}{\partial x_2^2}\right)\\
&=\Lambda\H_2.
 \end{split}
\end{equation*}
Clearly \eqref{13} implies the dissipations of $\H_2$ and $\omega$ (or equivalently of $\partial_{x_1}\H$ and $\omega$ due to $\Dv\H=0$); while \eqref{12} yields the dissipation for 
$b+\H_1$. 

\subsection{Dissipations of the density and the magnetic filed}
The goal of this subsection is to obtain the dissipations of the density and the magnetic field $\H_1$ from
\eqref{11}. To this end, we need to probe the relation between the
flow and the magnetic field.

Let us define the flow map $x(t,\alpha)$
associated to the velocity $\u$ as
\begin{equation}\label{lang}
\f{dx(t,\alpha)}{dt}=\u(t, x(t,\alpha))\quad
\textrm{with}\quad x(0)=\alpha.
\end{equation}
We introduce the deformation gradient $\F\in M^{2\times 2}$
($M^{2\times 2}$ denotes the set of all $2\times 2$ matrices with
positive determinants) as (see \cite{CD, LLZ1, LLZ, LZ} and references therein)
$$\F(t,x(t,\alpha))=\f{\partial x(t,\alpha)}{\partial\alpha}.$$
From the chain rule, it follows that $\F$ satisfies a transport
equation in the Eulerian coordinate
\begin{equation}\label{f}
\partial_t\F+\u\cdot\nabla\F=\nabla\u\F.
\end{equation}
Denote by $A$ and $J$ the inverse and the determinant of $\F=\nabla_\alpha x$ respectively; that is
$$A=\F^{-1}\quad\textrm{and}\quad J=\det\F.$$
Since $A\F=I$, differentiating $A$ gives
\begin{equation}\label{17a}
 \f{D}{Dt} A=-A\nabla_\alpha \u A\quad\textrm{and}\quad \partial_{\alpha_i}A=-A\nabla_\alpha\partial_{\alpha_i}x A,
\end{equation}
where $\f{D}{Dt}$ stands for the material derivative.
Differentiating $J$ gives
\begin{equation}\label{17b}
 \f{D}{Dt}J=J\tr(A\nabla_\alpha\u)\quad\textrm{and}\quad \partial_{\alpha_i}J=J\tr(A\nabla_\alpha\partial_{\alpha_i}x).
\end{equation}

The continuity equation equals
$$\partial_t\r+\u\cdot\nabla\r+\r\Dv\u=0,$$
and in Lagrangian coordinates, it reads
\begin{equation*}
 \f{D}{Dt}\r(t, x(t,\alpha))+\r(t,x(t,\alpha))\nabla_\alpha\u(t,x(t,\alpha)):A^\top=0.
\end{equation*}
This, together with \eqref{17b}, yields
\begin{equation*}
\r J=\r_0 J_0.
\end{equation*}
For simplicity of the presentation, we assume from now on that
\begin{equation}\label{19}
\r_0 J_0=1, \quad\textrm{and hence} \quad \r J=1 \quad\textrm{for all time}. 
\end{equation}

The magnetic field is incorporated into the flow through the deformation gradient as follows.
\begin{Proposition}\label{p11}
Assume that $(\r,\u, {\bf B})$ is a solution of \eqref{e1} and $\F$ satisfies the equation \eqref{f}. Then one has the relation
\begin{equation}\label{bf}
\r^{-1}A{\bf B}(t)=\r_0^{-1}A_0{\bf B_0}\quad \textrm{for all}\quad t\ge 0.
\end{equation}
\end{Proposition}
\begin{proof}
In Eulerian coordinates, \eqref{17a} is interpreted as
\begin{equation}\label{17d}
 \partial_t A+\u\cdot\nabla A+A\nabla\u=0.
\end{equation}
On the other hand, from the first and third equation of \eqref{e1}, one has
\begin{equation}\label{17e}
 \partial_t\left(\f{{\bf B}}{\r}\right)+\u\cdot\nabla\left(\f{{\bf B}}{\r}\right)=\left(\f{{\bf B}}{\r}\right)\cdot\nabla\u,
\end{equation}
which is refered as a ``frozen'' law of MHD in literatures (for example \cite{Ca,KL}).
Therefore one deduces from \eqref{17d} and \eqref{17e} that
$$\partial_t\Big(\r^{-1}A{\bf B}\Big)+\u\cdot\nabla\Big(\r^{-1}A{\bf B}\Big)=0,$$
and the desired identity \eqref{bf} follows.
\end{proof}

A direct consequence of Proposition \ref{p11} is that along the flow map, the quantity $\r^{-1}A{\bf B}$ is a constant. From now on, we assume that
\begin{equation}\label{A1}
\r^{-1}_0A_0{\bf B_0}=h_0.
\end{equation}
Therefore, it follows from Proposition \ref{p11} that for all $t>0$,
\begin{equation}\label{A11}
\r^{-1} A{\bf B}(t)=h_0.
\end{equation}
Keeping in mind that $A\F=I$ and \eqref{19}, one has
$$A_{ij}=\f{\partial \alpha_i}{\partial x_j}\quad \textrm{and}\quad \F=A^{-1}=\det \F\left[\begin{array}{ccc} A_{22},\quad -A_{12}\\ -A_{21},\quad A_{11}\end{array}\right]$$
Multiplying the identity \eqref{A11} by $\r\F$ implies
\begin{equation}\label{16}
\begin{split}
{\bf{B}}(t)=\r\F(x,t)h_0=\left(\begin{array}{ccc}A_{22}\\-A_{21}\end{array}\right).
\end{split}
\end{equation}
Introduce the perturbation of $A$ as
$$\mathcal{A}=A-I.$$
Components of \eqref{16} gives
\begin{equation}\label{16a}
\H_1=\mathcal{A}_{22}\quad\textrm{and}\quad \H_2=-\mathcal{A}_{21}. 
\end{equation}

Let us now explain the idea to obtain the dissipation for the magnetic field. Indeed the incompressibility $\Dv\H=0$ and the disspation of $\H_2$ imply the dissipation for $\partial_{x_1}\H$. On the other hand,
it holds, using \eqref{16a} and integration by parts
\begin{equation}\label{17}
 \begin{split}
(\partial_{x_1}\H|\Dv \mathcal{A})&=\sum_{j=1}^2\left(\f{\partial^2\alpha_2}{\partial x_1\partial x_2}|\f{\partial^2 \alpha_1}{\partial x_j^2}\right)
-\sum_{j=1}^2\left(\f{\partial^2\alpha_2}{\partial x_1^2}|\f{\partial^2 \alpha_2}{\partial x_j^2}\right)\\
&=-\sum_{j=1}^2\left\|\f{\partial^2\alpha_2}{\partial x_2\partial x_j}\right\|_{L^2}^2-\sum_{j=1}^2\left\|\f{\partial^2\alpha_2}{\partial x_1\partial_{x_j}}\right\|_{L^2}^2
+\sum_{i,j=1}^2\left(\f{\partial^2\alpha_i}{\partial x_i\partial x_j}|\f{\partial^2 \alpha_2}{\partial x_2\partial x_j}\right)\\
&=-\|\nabla\H\|_{L^2}^2+\sum_{i,j=1}^2\left(\f{\partial^2\alpha_i}{\partial x_i\partial x_j}|\f{\partial^2 \alpha_2}{\partial x_2\partial x_j}\right).
 \end{split}
\end{equation}
To control the last term above, we note that the identity \eqref{19} implies $\r=\det A$ and hence
\begin{equation}\label{16b}
b=\tr\mathcal{A}+\det\mathcal{A}\quad\textrm{or equivalently}\quad \tr\mathcal{A}=b-\det\mathcal{A}.
\end{equation}
Thus one has
\begin{equation*}
 \begin{split}
\sum_{i,j=1}^2\left(\f{\partial^2\alpha_i}{\partial x_i\partial x_j}|\f{\partial^2 \alpha_2}{\partial x_2\partial x_j}\right)=
\sum_{j=1}^2\left(\f{\partial b}{\partial x_j}|\f{\partial^2 \alpha_2}{\partial x_2\partial x_j}\right)
-\sum_{j=1}^2\left(\f{\partial\det\mathcal{A}}{\partial x_j}|\f{\partial^2 \alpha_2}{\partial x_2\partial x_j}\right).
 \end{split}
\end{equation*}
While the second term in the right hand side above is a cubic term essentially, the first term in the right hand side is quadratic and hence needs to deal with it more carefully. Actually there are two cases for the
quadratic term
\begin{itemize}
 \item $\left(\f{\partial b}{\partial x_1}|\f{\partial^2 \alpha_2}{\partial x_2\partial x_1}\right)=\left(\f{\partial (b+\H_1)}{\partial x_1}|\f{\partial^2 \alpha_2}{\partial x_2\partial x_1}\right)
-\left\|\f{\partial^2 \alpha_2}{\partial x_2\partial x_1}\right\|_{L^2}^2=\left(\f{\partial (b+\H_1)}{\partial x_1}|\f{\partial^2 \alpha_2}{\partial x_2\partial x_1}\right)-\|\partial_{x_1}\H_1\|_{L^2}^2$; and
\item $\left(\f{\partial b}{\partial x_2}|\f{\partial^2 \alpha_2}{\partial x_2\partial x_2}\right)=\left(\f{\partial (b+\H_1)}{\partial x_2}|\f{\partial^2 \alpha_2}{\partial x_2\partial x_2}\right)
-\left\|\f{\partial^2 \alpha_2}{\partial x^2_2}\right\|_{L^2}^2=\left(\f{\partial (b+\H_1)}{\partial x_2}|\f{\partial^2 \alpha_2}{\partial x_2\partial x_2}\right)-\|\partial_{x_2}\H_1\|_{L^2}^2$.
\end{itemize}
Note that the dissipation for $\nabla(b+\H_1)$ is clear from \eqref{12}. Thus substituting those two identities into \eqref{17} yields
\begin{equation}\label{18a}
 \begin{split}
\|\nabla\H\|_{L^2}^2+\left\|\nabla\H_1\right\|_{L^2}^2&=-(\partial_{x_1}\H|\Dv \mathcal{A})+\sum_{i=1}^2\left(\f{\partial (b+\H_1)}{\partial x_i}|\f{\partial^2 \alpha_2}{\partial x_2\partial x_i}\right)\\
&\quad-\sum_{j=1}^2\left(\f{\partial\det\mathcal{A}}{\partial x_j}|\f{\partial^2 \alpha_2}{\partial x_2\partial x_j}\right).
 \end{split}
\end{equation}

On the other hand, integration by parts gives
\begin{equation}\label{18b}
 \begin{split}
(\nabla(b+\H_1)|\Dv \mathcal{A})&=-(b+\H_1|\D \tr\mathcal{A})=-(b+\H_1|\D(b-\det\mathcal{A}))\\
&=\|\nabla(b+\H_1)\|^2_{L^2}-(\nabla(b+\H_1)|\nabla\H_1)-(\nabla(b+\H_1)|\nabla\det\mathcal{A}).
 \end{split}
\end{equation}
Adding \eqref{18a} and \eqref{18b} together yields
\begin{equation}\label{18}
 \begin{split}
\|\nabla\H\|_{L^2}^2+\|\nabla b\|_{L^2}^2=(\nabla(b+\H_1)-\partial_{x_1}\H|\Dv \mathcal{A})+(\nabla b|\nabla\det\mathcal{A}).
 \end{split}
\end{equation}
All terms in the right hand side contain at least one term with $L^1$ dissipation in time, provided that one can obtain the dissipation for $\det\mathcal{A}$, which is a quadratic term and contains $\H$. 
This in turns gives a $L^2$ dissipation in time for $\H$ and $b$, and hence a $L^1$ dissipation for nonlinear terms such as $|\H_1|^2$ and $b^2$.

\begin{Remark}
 The reason why we prefer to using $A$ instead of the deformation gradient $\F$ itself lies in the fact that $A$ is actually a gradient in Eulerian coordinates, and hence taking spatial derivatives of $A$ will not involve 
change of variables.  
\end{Remark}

\subsection{Spectrum Analysis.} Recall that the linearization of \eqref{e1} takes the form \eqref{e3a}
\begin{equation*}
\partial_{tt}\u-\D\u-Q(D)\u=0,
\end{equation*}
with the characteristic equation
\begin{equation}\label{81a1}
\eta^2I+|\xi|^2\eta I-Q(\xi)=0\quad \textrm{where}\quad \textrm{$I$ is the $2\times 2$ identity matrix.}
\end{equation}
To control the competition between the parabolicity and hyperbolicity of \eqref{81a1}, one needs to analyze the positive definiteness of the symmetric matrix $-Q(\xi)$. The matrix $-Q(\xi)$ has two eigenvalues
$$\lambda_\pm=|\xi|^2\left(1\pm\sqrt{1-\left(\f{\xi_1}{|\xi|}\right)^2}\right),$$ and the equation \eqref{81a1} has the discriminant
$$|\xi|^4I+4Q(\xi).$$ 
Eigenvalues $\lambda_\pm$ are called fast and slow magneto-acoustic wave speed (cf \cite{Ca, KL}) respectively.
Note that whenever the discriminant is positive definite the equation \eqref{81a1} (and \eqref{e3a}) behaves like the hyperbolic structure, 
and if the discriminant is negative definite the equation \eqref{81a1} behaves like a parabolic
system. The worst situation is the case when the discriminant is neither positive nor negative definite; and under this situation one component of the system \eqref{81a1} has the parabolic feature and the other component 
has the hyperbolic structure. To handle the complexity of \eqref{81a1}, it will be convenient to diagonalize $-Q(\xi)$ (and hence \eqref{81a1}).

Since $-Q(\xi)$ is symmetric, there exists an orthonormal matrix $\mathcal{P}(\xi)$ such that
$$-\mathcal{P}^\top(\xi) Q(\xi)\mathcal{P}(\xi)=\textrm{diag}(\lambda_{-},\lambda_{+}).$$ A direct computation, combining with Taylor's expansion
$$\sqrt{1+x}=1+\f12x+\sum_{n=2}^\infty\left(\begin{array}{ccc}1/2\\ n\end{array}\right)x^n\quad\textrm{as}\quad |x|<1,$$ gives
\begin{equation*}
 \begin{split}
\mathcal{P}(\xi)&=\f{1}{\sqrt{\xi_2^2+\xi_1^2\Big(\f12+r(\xi)\Big)^2}}\left(\begin{array}{ccc}-\xi_2&\xi_1\Big(\f12+r(\xi)\Big)\\\xi_1\Big(\f12+r(\xi)\Big)&\xi_2\end{array}\right)
 \end{split}
\end{equation*}
with $$r(\xi)=\sum_{n=2}^\infty\left(\begin{array}{ccc}1/2\\ n\end{array}\right)(-1)^n\left(\f{\xi_1^{2}}{|\xi|^{2}}\right)^{n-1}.$$ Observe that $\mathcal{P}(\xi)=\mathcal{P}(\xi)^\top$ and $\mathcal{P}(\xi)^2=I$.
Denote by $\mathcal{P}, \lambda_{\pm}(D)$ the differential operators with the symbols $\mathcal{P}(\xi), \lambda_{\pm}(\xi)$ respectively.

Note that \eqref{16a} and \eqref{16b} imply
\begin{equation*}
 \begin{split}
\nabla(b+\H_1)-\partial_{x_1}\H=Q(D)\alpha+\nabla\det\mathcal{A},
 \end{split}
\end{equation*}
and hence
\begin{equation}\label{16d}
 \begin{split}
\mathcal{P}\Big(\nabla(b+\H_1)-\partial_{x_1}\H\Big)=-\textrm{diag}(\lambda_{-}(D), \lambda_{+}(D))\mathcal{P}\alpha+\mathcal{P}\nabla\det\mathcal{A}.
 \end{split}
\end{equation}
From now on, we denote
\begin{equation*}
\mathcal{H}=\mathcal{P}\Big(\nabla(b+\H_1)-\partial_{x_1}\H\Big). 
\end{equation*}

\subsection{Main results}
Let us first introduce the functional space which appears in the
global existence theorem. Let $T>0$, $r\in[0,\infty]$ and $X$ be a Banach space. We denote by
$\mathcal{M}(0,T;X)$ the set of measurable functions on $(0,T)$
valued in $X$. For $f\in \mathcal{M}(0,T; X)$, we define
\begin{subequations}
 \begin{align}
&\|f\|_{L^r_T(X)}=\left(\int_0^T\|f(\tau)\|_X^rd\tau\right)^{\f{1}{r}}\quad\textrm{
if }\quad r<\infty,\nonumber\\
&\|f\|_{L^\infty_T(X)}=\sup \textrm{ess}_{\tau\in(0,T)}\|f(\tau)\|_X. \nonumber 
 \end{align}
\end{subequations}
Denote $L^r(0,T;X)=\{f\in
\mathcal{M}(0,T;X)|\|f\|_{L^r_T(X)}<\infty\}$. If $T=\infty$, we
denote by $L^r(\R^+; X)$ and $\|f\|_{L^r(X)}$ the corresponding
spaces and norms. Also denote by $C([0,T],X)$ (or $C(\R^+,X)$) the
set of continuous X-valued functions on $[0,T]$ (resp. $\R^+$).

Now we are ready to state our main theorem. 
\begin{Theorem}\label{MT}
There exist two positive constants $\gamma$ and $\Gamma$,
such that, if $\u_0\in
\hat{B}^{0}$, $\lambda_{-}^{-\f12}(D)\mathcal{H}(0)_1\in\tilde{B}^{0,1}$, $\lambda_{+}^{-\f12}(D)\mathcal{H}(0)_2\in\hat{B}^{0}\cap\hat{B}^1$,  and $\mathcal{A}_0\in\hat{B}^{1}$
satisfy
\begin{itemize}
\item $\|\u_0\|_{\hat{B}^{0}}+\left\|\lambda_{-}^{-\f12}(D)\mathcal{H}(0)_1\right\|_{\tilde{B}^{0,1}}+\left\|\lambda_{+}^{-\f12}(D)\mathcal{H}(0)_2\right\|_{\hat{B}^{0}\cap\hat{B}^1}+\|\mathcal{A}_0\|_{\hat{B}^{1}}\le
\gamma$ for a sufficiently small $\gamma$;
\item $\r_0$, ${\bf{B}_0}$, and $\F_0$ satisfies conditions \eqref{19} and \eqref{A1},
\end{itemize}
then Cauchy problem \eqref{e1} with initial data \eqref{IC1} has a unique global solution $(\r,\u,{\bf B})$ which
satisfies the following estimate
\begin{equation}\label{ES}
\begin{split}
 &\|\mathcal{A}\|_{L^\infty(\hat{B}^{1})}+\|\u\|_{L^\infty(\hat{B}^{0})}+\left\|\lambda_{-}^{-\f12}(D)\mathcal{H}_1\right\|_{L^\infty(\tilde{B}^{0,1})}
+\left\|\lambda_{+}^{-\f12}(D)\mathcal{H}_2\right\|_{L^\infty(\hat{B}^{0}\cap\hat{B}^1)}\\
&\qquad+\|\u\|_{L^1(\hat{B}^{2})}+\|b\|_{L^2(\hat{B}^{1})}+\|\H\|_{L^2(\hat{B}^{1})}\\
&\quad\le
\Gamma\left(\|\mathcal{A}_0\|_{L^\infty(\hat{B}^{1})}+\|\u_0\|_{\hat{B}^{0}}+\left\|\lambda_{-}^{-\f12}(D)\mathcal{H}(0)_1\right\|_{\tilde{B}^{0,1}}+\left\|\lambda_{+}^{-\f12}(D)\mathcal{H}(0)_2\right\|_{\hat{B}^{0}\cap\hat{B}^1}
\right).
\end{split}
\end{equation}
\end{Theorem}

We refer the reader to Section 3 below for the definiton of \textit{hybrid Besov space} $\tilde{B}^{s,t}$.

Provided that the energy estimate \eqref{ES} holds true, Theorem \ref{MT} is verified by extending a local solution. The local wellposedness with sufficiently smooth data had been
considered in \cite{KW}. For general data, a fixed-point argument will be sufficient to construct a local existence (see for instance \cite{QZ}). From this standpoint, we will focus on in this paper
how to obtain the estimate \eqref{ES}. 

\texttt{Strategy of Proof.}  In order to show that a local solution can be extended to be a global one, we only need to prove 
the uniform estimate \eqref{ES}. For this purpose, we denote
\begin{equation*}
 \begin{split}
X(t)&=\|\mathcal{A}\|_{L^\infty_t(\hat{B}^{1})}+\|\u\|_{L_t^\infty(\hat{B}^{0})}+\left\|\lambda_{-}^{-\f12}(D)\mathcal{H}_1\right\|_{L_t^\infty(\tilde{B}^{0,1})}
+\left\|\lambda_{+}^{-\f12}(D)\mathcal{H}_2\right\|_{L_t^\infty(\hat{B}^{0}\cap\hat{B}^1)}\\
&\qquad+\|\u\|_{L_t^1(\hat{B}^{2})}+\|b\|_{L_t^2(\hat{B}^{1})}+\|\H\|_{L_t^2(\hat{B}^{1})},  
 \end{split}
\end{equation*}
and we are going to show
\begin{equation}\tag{$\mathfrak{G}$}
X(t)\le C\Big(X(0)+X^2(t)\Big).
\end{equation}
Once ($\mathfrak{G}$) was shown, the existence part of Theorem \ref{MT} is done since by the continuity of $X(t)$ and the smallness of the initial data, there exists a constant $\Gamma$ such that
$$X(t)\le \Gamma X(0),$$ and hence local solutions can be extended.

\begin{Remark} Three remarks go as follows:
\begin{itemize}
 \item 
In spirit of the method in \cite{RD1}, one can verify that the solution in Theorem \ref{MT} is unique, but for the clarity of the global existence issue, we omit the proof of the
uniqueness here.
\item Theorem \ref{MT} says that the global wellposedness of MHD with zero magnetic diffusivity strongly depends on the flow. In other words, to ensure the global wellposedness,
not only the perturbation is small, the coupling between the flow and the magnetic field need to be very subtle (see Proposition \ref{p11} and \eqref{19}).
\item Theorem \ref{MT} and its strategy of proof can be extended to the three-dimensional
case with a slight modification for the regularity of \textit{hybrid Besov spaces.}
\end{itemize}
\end{Remark}

Theorem \ref{MT} coincides with results in \cite{HW2} and \cite{RD} for the incompressible model of MHD and compressible Navier-Stokes equations respectively. Indeed, one has the following estimate.
\begin{Lemma}\label{le}
Assume that $\|\mathcal{A}\|_{\hat{B}^1}\le \i$ for sufficiently small $\i$. Then
\begin{subequations}\label{le1}
 \begin{align}
\|b-\H_1\|_{\tilde{B}^{0,1}}+\|\H_2\|_{\tilde{B}^{0,1}}\lesssim \left\|\lambda_{-}^{-\f12}(D)\mathcal{H}_1\right\|_{\tilde{B}^{0,1}}+\left\|\lambda_{+}^{-\f12}(D)\mathcal{H}_2\right\|_{\hat{B}^0\cap\hat{B}^1},\label{le1b}\\
\|b\|_{\hat{B}^0\cap\hat{B}^1}\lesssim \left\|\lambda_{-}^{-\f12}(D)\mathcal{H}_1\right\|_{\tilde{B}^{0,1}}+\left\|\lambda_{+}^{-\f12}(D)\mathcal{H}_2\right\|_{\hat{B}^0\cap\hat{B}^1}.\label{le1a}
 \end{align}
\end{subequations}
\end{Lemma}
\begin{proof}
From \eqref{16d}, it follows
\begin{equation}\label{le2}
 \alpha=-\mathcal{P}\Big(\lambda_{-}^{-1}(D)\mathcal{H}_1, \lambda_{+}^{-1}(D)\mathcal{H}_2\Big)^\top+\mathcal{P}\textrm{diag}(\lambda_{-}^{-1}(D), \lambda_{+}^{-1}(D))\mathcal{P}\nabla\det\mathcal{A}.
\end{equation}
Note that the symbol associated with the differential operator $\mathcal{P}\nabla$ is
\begin{equation}\label{le22}
\f{1}{\sqrt{\xi_2^2+\xi_1^2\left(\f12+r(\xi)\right)^2}}\left(\xi_1\xi_2\Big(-\f12+r(\xi)\Big),\xi_1^2\Big(\f12+r(\xi)\Big)+\xi^2_2\right),
\end{equation}
and thus
\begin{subequations}\label{le21}
\begin{align}
&\mathcal{F}\Big(\mathcal{P}\textrm{diag}(\lambda_{-}^{-1}(D), \lambda_{+}^{-1}(D))\mathcal{P}\nabla\det\mathcal{A}\Big)_1
\lesssim \left|\lambda_{-}^{-\f12}(\xi)\mathcal{F}(\det\mathcal{A})\right|+\left|\lambda_{+}^{-\f12}(\xi)\mathcal{F}(\det\mathcal{A})\right|\label{le21a}\\
&\mathcal{F}\Big(\mathcal{P}\textrm{diag}(\lambda_{-}^{-1}(D), \lambda_{+}^{-1}(D))\mathcal{P}\nabla\det\mathcal{A}\Big)_2
\lesssim |\xi|^{-1}\left|\mathcal{F}(\det\mathcal{A})\right|\label{le21b}
\end{align}
\end{subequations}
since $\lambda_{-}(\xi)\approx \xi_1^2$ and $\lambda_{+}(\xi)\approx |\xi|^2$.

\texttt{Estimate of \eqref{le1b}.} Identities \eqref{16a} and \eqref{16b} imply
$$b-\H_1=\partial_{x_1}\alpha_1+\det\mathcal{A}\quad\textrm{and}\quad\H_2=-\partial_{x_1}\alpha_2.$$
This, combining with \eqref{le2} and \eqref{le21a}, gives
\begin{equation*}
 \|\hat{\D}_{q,k}(b-\H_1)\|_{L^2}+\|\hat{\D}_{q,k}\H_2\|_{L^2}\lesssim 2^k\|\hat{\D}_{q,k}\alpha\|_{L^2}+\|\hat{\D}_{q,k}\det\mathcal{A}\|_{L^2},
\end{equation*}
and hence one deduces from \eqref{le2} that
\begin{equation}\label{le4}
\begin{split}
\|b-\H_1\|_{\tilde{B}^{0,1}}+\|\H_2\|_{\tilde{B}^{0,1}}\lesssim \left\|\lambda_{-}^{-\f12}(D)\mathcal{H}_1\right\|_{\tilde{B}^{0,1}}+\left\|\lambda_{+}^{-\f12}(D)\mathcal{H}_2\right\|_{\hat{B}^0\cap\hat{B}^1}
+\|\det\mathcal{A}\|_{\tilde{B}^{0,1}}.
\end{split}
\end{equation}

Note that the identities \eqref{16a} and \eqref{16b} imply
\begin{equation}\label{le6}
 \begin{split}
\det\mathcal{A}&=\partial_{x_1}\alpha_1\partial_{x_2}\alpha_2-\partial_{x_1}\alpha_2\partial_{x_2}\alpha_1\\
&=(b-\H_1)\partial_{x_2}\alpha_2+\H_2\partial_{x_2}\alpha_1+\det\mathcal{A}\partial_{x_2}\alpha_2.
 \end{split}
\end{equation}
Thus Lemma \ref{lee} yields
$$\|\det\mathcal{A}\|_{\tilde{B}^{0,1}}\lesssim \|b-\H_1\|_{\tilde{B}^{0,1}}\|\mathcal{A}\|_{\hat{B}^1}+\|\H_2\|_{\tilde{B}^{0,1}}\|\mathcal{A}\|_{\hat{B}^1}+\|\mathcal{A}\|_{\hat{B}^1}\|\det\mathcal{A}\|_{\tilde{B}^{0,1}},$$
which gives, if $\|\mathcal{A}\|_{\hat{B}^1}\le \i$ for sufficiently small $\i$
\begin{equation}\label{le3}
\|\det\mathcal{A}\|_{\tilde{B}^{0,1}}\lesssim \|b-\H_1\|_{\tilde{B}^{0,1}}\|\mathcal{A}\|_{\hat{B}^1}+\|\H_2\|_{\tilde{B}^{0,1}}\|\mathcal{A}\|_{\hat{B}^1}. 
\end{equation}
The desired estimate \eqref{le1b} follows from \eqref{le4}, \eqref{le3} and the assumption that $\|\mathcal{A}\|_{\hat{B}^1}$ is sufficiently small.

\texttt{Estimate of \eqref{le1a}.} The identity \eqref{16b} implies that
$$\|\hat{\D}_{q,k}b\|_{L^2}\lesssim \|\hat{\D}_{q,k}\Dv\alpha\|_{L^2}+\|\hat{\D}_{q,k}\det\mathcal{A}\|_{L^2},$$
which, combining \eqref{le21}, gives
\begin{equation*}
\|\hat{\D}_{q,k}b\|_{L^2}\lesssim \|\hat{\D}_{q,k}\lambda_{-}^{-\f12}(D)\mathcal{H}_1\|_{L^2}+\|\hat{\D}_{q,k}\lambda_{+}^{-\f12}(D)\mathcal{H}_2\|_{L^2}+\|\hat{\D}_{q,k}\det\mathcal{A}\|_{L^2}.
\end{equation*}
Thus it follows from the identity above and the inequality \eqref{25a} that
\begin{equation}\label{le5}
\begin{split}
\|b\|_{\hat{B}^0\cap\hat{B}^1}&\le \|\lambda_{-}^{-\f12}(D)\mathcal{H}_1\|_{\hat{B}^0\cap\hat{B}^1}+\|\lambda_{+}^{-\f12}(D)\mathcal{H}_2\|_{\hat{B}^0\cap\hat{B}^1}+\|\det\mathcal{A}\|_{\hat{B}^0\cap\hat{B}^1}\\
&\lesssim \|\lambda_{-}^{-\f12}(D)\mathcal{H}_1\|_{\tilde{B}^{0,1}}+\|\lambda_{+}^{-\f12}(D)\mathcal{H}_2\|_{\hat{B}^0\cap\hat{B}^1}+\|\det\mathcal{A}\|_{\hat{B}^0\cap\hat{B}^1}
\end{split}
\end{equation}

The identity \eqref{le6}, Proposition \ref{p2} and the iequality \eqref{25a} imply
\begin{equation*}
\begin{split}
\|\det\mathcal{A}\|_{\hat{B}^0\cap\hat{B}^1}&\le \|b-\H_1\|_{\hat{B}^0\cap\hat{B}^1}\|\mathcal{A}\|_{\hat{B}^1}+\|\H_2\|_{\hat{B}^0\cap\hat{B}^1}\|\mathcal{A}\|_{\hat{B}^1}
+\|det\mathcal{A}\|_{\hat{B}^0\cap\hat{B}^1}\|\mathcal{A}\|_{\hat{B}^1}\\
&\lesssim \|b-\H_1\|_{\tilde{B}^{0,1}}\|\mathcal{A}\|_{\hat{B}^1}+\|\H_2\|_{\tilde{B}^{0,1}}\|\mathcal{A}\|_{\hat{B}^1}
+\|det\mathcal{A}\|_{\hat{B}^0\cap\hat{B}^1}\|\mathcal{A}\|_{\hat{B}^1},
\end{split}
\end{equation*}
and thus
\begin{equation}\label{le7}
 \|\det\mathcal{A}\|_{\hat{B}^0\cap\hat{B}^1}\lesssim \|b-\H_1\|_{\tilde{B}^{0,1}}\|\mathcal{A}\|_{\hat{B}^1}+\|\H_2\|_{\tilde{B}^{0,1}}\|\mathcal{A}\|_{\hat{B}^1}
\end{equation}
if $\|\mathcal{A}\|_{\hat{B}^1}$ is sufficiently small.

The desired estimate \eqref{le1a} follows from \eqref{le5}, \eqref{le7}, \eqref{le1b} and the smallness of $\|\mathcal{A}\|_{\hat{B}^1}.$
\end{proof}

As a direct consequence of Lemma \ref{le} and the inequality \eqref{25a}, it follows
\begin{equation}\label{le8}
 \begin{split}
\|b\|_{\hat{B}^0\cap\hat{B}^1}+\|\H\|_{\hat{B}^0\cap\hat{B}^1}\lesssim \left\|\lambda_{-}^{-\f12}(D)\mathcal{H}_1\right\|_{\tilde{B}^{0,1}}+\left\|\lambda_{+}^{-\f12}(D)\mathcal{H}_2\right\|_{\hat{B}^0\cap\hat{B}^1}.  
 \end{split}
\end{equation}

\bigskip\bigskip
%%%%%%%%%%%%%%%%%%%%%%%%%%%%%%%%%%%%%%%%%%%%%%%%%%%%%%%%%%%%%%%%%%%%%%%%%%%%%%%%%%%%%%%%%%%%%%%%%%%%%%%%%%%%%%%%%%%%%%%%%%

\section{Besov Spaces}

Throughout this paper,  we use $C$ for a generic constant, and
denote  $a\le Cb$ by  $a\lesssim b$. The notation $a\thickapprox
b$ means that $a\lesssim b$ and $b\lesssim a$. Also we use
$(\alpha_{q,k})_{q,k\in\mathbb{Z}}$ to denote a sequence such that
$\sum_{q,k\in\mathbb{Z}}\alpha_{q,k}\le 1$. The standard
summation notation over the repeated index is adopted in this
paper.

The definition of the homogeneous Besov space is built on an
homogeneous Littlewood-Paley decomposition. First, we introduce a
function $\psi\in C^\infty(\R^2)$, supported in
$\mathcal{C}=\{\xi\in\R^2: \f{5}{6}\le|\xi|\le\f{12}{5}\}$ and
such that
$$\sum_{q\in\mathbb{Z}}\psi(2^{-q}\xi)=1\textrm{ if }\xi\neq 0.$$
Denoting $\mathcal{F}^{-1}\psi$ by $h$, we define the dyadic
blocks as follows:
$$\D_q f=\psi(2^{-q}D)f=2^{2q}\int_{\R^2}h(2^qy)f(x-y)dy,$$
and
$$S_q f=\sum_{p\le q-1}\D_pf.$$
The formal decomposition
\begin{equation}\label{21}
f=\sum_{q\in\mathbb{Z}}\D_qf
\end{equation}
is called homogeneous Littlewood-Paley decomposition in $\R^2$. Similarly, we use $\D_k^1$ to denote the homogeneous Littlewood-Paley decomposition in $\R$ in the direction of $x_1$.

For $s\in\R$ and $f\in \mathcal{S}'(\R^2)$, we denote
$$\|f\|_{\dot{B}^s_{p,r}}\overset{def}{=}\left(\sum_{q\in\mathbb{Z}}2^{sqr}\|\D_qf\|^r_{L^p}\right)^{\f{1}{r}}.$$
As $p=2$ and $r=1$, we denote $\|\cdot\|_{\dot{B}^s_{p,r}}$ by
$\|\cdot\|_{B^s}$.
\begin{Definition}
Let $s\in\R$, and $m=-\left[2-s\right]$. If $m<0$, we set
$$B^s=\left\{f\in \mathcal{S}'(\R^2)|\|f\|_{B^s}<\infty\textrm{
and }f=\sum_{q\in\mathbb{Z}}\D_qf\textrm{ in
}\mathcal{S}'(\R^2)\right\}.$$ If $m\ge 0$, we denote by
$\mathcal{P}_m$ the set of two variables polynomials of degree
$\le m$ and define
$$B^s=\left\{f\in \mathcal{S}'(\R^2)/\mathcal{P}_m|\|f\|_{B^s}<\infty\textrm{
and }f=\sum_{q\in\mathbb{Z}}\D_qf\textrm{ in
}\mathcal{S}'(\R^2)/\mathcal{P}_m\right\}.$$
\end{Definition}

Functions in $B^s$ has many good properties (see Proposition 2.5
in \cite{RD}):
\begin{Proposition}\label{p3}
The following properties hold:
\begin{itemize}
\item Derivation: $\|f\|_{B^s}\thickapprox \|\nabla f\|_{B^{s-1}}$;
\item Fractional derivation: let $\Gamma=\sqrt{-\D}$ and
$\sigma\in\R$; then the operator $\Gamma^\sigma$ is an isomorphism
from $B^s$ to $B^{s-\sigma}$;
\item Algebraic properties: for $s>0$, $B^s\cap L^\infty$ is an
algebra.
\end{itemize}
\end{Proposition}

To handle the degeneracy of the hyperbolicity, we further apply the Littlewood-Paley decomposition in $x_1$ to introduce a smaller Besov space $\hat{B}^s$ as
\begin{Definition}
Let $q, k\in \mathbb{Z}$ and $s\in\R$. We set
$$\|f\|_{\hat{B}^{s}}=\sum_{q,k\in\mathbb{Z}}2^{qs}\|\hat{\D}_{q,k}f\|_{L^2}\quad \textrm{with}\quad \hat{\D}_{q,k}\overset{def}=\D_q\D_{k}^1.$$
The Besov space $\hat{B}^{s}$ is defined by
$$\hat{B}^{s}=\left\{f\in\mathcal{S}'(\R^2)|\|f\|_{\hat{B}^{s}}<\infty\right\}.$$
Note that $\hat{B}^s\subset B^s$ for $s\in\R$ since $\ell_1\subset \ell_2$. In particular $\hat{B}^1\subset B^1\subset L^\infty$.
\end{Definition}
It is implicitly indicated in the double Littlewood-Paley decomposition $\hat{\D}_{q,k}$ that $k\le q$ since $|\xi_1|\le |\xi|$.

The product in Besov space $\hat{B}^s$ can be estimated by
\begin{Proposition}\label{p2}
For all $s, t\le 1$ such that $s+t>0$,
$$\|fg\|_{\hat{B}^{s+t-1}}\lesssim
\|f\|_{\hat{B}^{s}}\|g\|_{\hat{B}^t}.$$
\end{Proposition}
The proof of Proposition \ref{p2} will be postponed to the Appendix.

To deal with functions with different regularities for different frequencies as suggested by the spectral analysis in Section 2, it is
more effective to work in a \textit{hybrid Besov space} $\tilde{B}^{s,t}$. 
\begin{Definition}
Let $q, k\in \mathbb{Z}$ and $s,t\in\R$. We set
$$\|f\|_{\tilde{B}^{s,t}}=\sum_{k+1\ge 2q}2^{qs}\|\hat{\D}_{q,k}f\|_{L^2}+\sum_{k+1< 2q}2^{(2q-k)t}\|\hat{\D}_{q,k}f\|_{L^2}.$$
The \textit{hybrid Besov space} $\tilde{B}^{s,t}$ is defined by
$$\tilde{B}^{s,t}=\left\{f\in\mathcal{S}'(\R^2)|\|f\|_{\tilde{B}^{s,t}}<\infty\right\}.$$
\end{Definition}

\begin{Remark}
Four remarks go as follows.
\begin{itemize}
 \item The space $\tilde{B}^{s,t}$ is not empty since for any function $\phi\in C^\infty(\R^2)$ with support of its fourier transform in the union of balls $\{\xi\in\R^2|2|\xi_1|\ge |\xi|^2\}$, 
it holds $\mathcal{F}\phi\in \tilde{B}^{s,t}.$
\item $\tilde{B}^{0,1}$ is continuously embedded into $L^\infty$. Indeed, for any fixed $q,k\in\mathbb{Z}$, it holds
\begin{equation*}
\begin{split}
\|\hat{\D}_{q,k}f\|_{L^\infty}&\le \|\mathcal{F}(\hat{\D}_{q,k}f)\|_{L^1}\\
&\lesssim
\begin{cases}
\|\mathcal{F}(\hat{\D}_{q,k}f)\|_{L^2}\textrm{meas}(\{2|\xi_1|\ge |\xi|^2\})\quad \textrm{if}\quad k+1\ge 2q;\\
\|\xi_1^{-1}|\xi|^2\mathcal{F}(\hat{\D}_{q,k}f)\|_{L^2}\||\xi_1||\xi|^{-2}\|_{L^2(\{2|\xi_1|<|\xi|^2\}\cap \{2^q\le |\xi|\le 2^{q+1}\})} \quad \textrm{otherwise}
\end{cases}\\
&\lesssim \max\{2, 2^{2q-k}\}\|\hat{\D}_{q,k}f\|_{L^2}\\
&\lesssim \|f\|_{\tilde{B}^{0,1}}.
\end{split}
\end{equation*}
Thus, $\sum_{q,k}\hat{\D}_{q,k}f$ uniformly converges to $f$ in $L^\infty$.
\item If $0\le s\le t$, $\tilde{B}^{s,t}\subset \hat{B}^s\cap\hat{B}^t$, that is
\begin{equation}\label{25a}
\|\phi\|_{\hat{B}^s\cap\hat{B}^t}=\max\{\|\phi\|_{\hat{B}^s},\|\phi\|_{\hat{B}^t}\}\lesssim \|\phi\|_{\tilde{B}^{s,t}}.
\end{equation}
\item Denote by $\check{B}^s$ the functional space with the norm
$$\|f\|_{\check{B}^s}\overset{def}=\sum_{q,k\in\mathbb{Z}}2^{(2q-k)s}\|\hat{\D}_{q,k}\phi\|_{L^2}.$$
It is easy to check that 
\begin{equation}\label{27a}
 \tilde{B}^{0,1}=\check{B}^0\cap\check{B}^1=\hat{B}^0\cap\check{B}^1.
\end{equation}
Moreover, similar as Proposition \ref{p2}, one can verify
\begin{equation}\label{27b}
\|fg\|_{\check{B}^1}\lesssim \|f\|_{\hat{B}^1}\|g\|_{\check{B}^1}.
\end{equation}
\end{itemize}
\end{Remark}
An estimate for the product in our hybrid Besov space $\tilde{B}^{0,1}$ is needed.
\begin{Lemma}\label{lee}
There holds true
$$\|fg\|_{\tilde{B}^{0,1}}\lesssim \|f\|_{\tilde{B}^{0,1}}\|g\|_{\hat{B}^1}.$$
\end{Lemma}
\begin{proof}
It is a direct consequence of Proposition \ref{p2}, \eqref{27a} and \eqref{27b}.
\end{proof}

Throughout this paper, the following estimates for the convective
terms arising in the localized system is used several times.
\begin{Lemma}\label{cl}
Let $G$ be a smooth function away from the origin with the form $|\xi|^m\xi_{1}^{n}$. 
Then there hold true
\begin{equation}\label{25b}
 \begin{split}
&|(G(D)\hat{\D}_{q,k}(e\cdot\nabla f)|G(D)\hat{\D}_{q,k} f)|\\
&\quad\le C\alpha_{q,k}
2^{nk+mq}
\|e\|_{\hat{B}^{2}}\|f\|_{\hat{B}^{0}}\left\|G(D)\hat{\D}_{q,k}
f\right\|_{L^2},  
 \end{split}
\end{equation}
\begin{equation}\label{25}
\begin{split}
&|(G(D)\hat{\D}_{q,k}(e\cdot\nabla f)|G(D)\hat{\D}_{q,k} f)|\\
&\quad\le C\alpha_{q,k}
2^{nk+mq}\min\{2^{-1}, 2^{k-2q}\}
\|e\|_{\hat{B}^{2}}\|f\|_{\tilde{B}^{0,1}}\left\|G(D)\hat{\D}_{q,k}
f\right\|_{L^2},
\end{split}
\end{equation}
and
\begin{equation}\label{26}
\begin{split}
&\left|(G(D)\hat{\D}_{q,k}(e\cdot\nabla f)|\hat{\D}_{q,k} g)+(\hat{\D}_{q,k}(e\cdot\nabla
g)|G(D)\hat{\D}_{q,k} f)\right|\\&\le
C\alpha_{q,k}\|e\|_{\hat{B}^{2}}\Big(\left\|G(D)\hat{\D}_{q,k}f\right\|_{L^2}\|g\|_{\hat{B}^0}\\&\qquad+2^{nk+mq}\min\{2^{-1},2^{k-2q}\}
\|f\|_{\tilde{B}^{0,1}}\|\hat{\D}_{q,k}g\|_{L^2}\Big),
\end{split}
\end{equation}
where $\sum_{q,k\in\mathbb{Z}}\alpha_{q,k}\le 1$.
\end{Lemma}
We postpone the proof of Lemma \ref{cl} to the Appendix. When the space $\tilde{B}^{0,1}$ is replaced by $\hat{B}^0\cap \hat{B}^1$, a variant version of Lemma \ref{cl} has been verified in \cite{RD} (see Lemma 5.1 there; 
see also \cite{CL} for the case when $e$ is divergence-free). 

The following commutator estimate will be useful in handling energy estimate, and its proof is again postponed to the Appendix.
\begin{Lemma}\label{cl1}
Let $\chi(\xi)$ be a smooth function away from the origin. Then there hold true
\begin{equation}\label{27c}
\|\chi(D)(\u\cdot\nabla f)-\u\cdot\nabla (\chi(D)f)\|_{\tilde{B}^{0,1}}\lesssim \|\chi(D)f\|_{\tilde{B}^{0,1}}\|\u\|_{\hat{B}^1} 
\end{equation}
and
\begin{equation}\label{27d}
\|\chi(D)(\u\cdot\nabla f)-\u\cdot\nabla (\chi(D)f)\|_{\hat{B}^{0}\cap \hat{B}^1}\lesssim \|\chi(D)f\|_{\hat{B}^{0}\cap \hat{B}^1}\|\u\|_{\hat{B}^1}.
\end{equation}
\end{Lemma}

\bigskip\bigskip
%%%%%%%%%%%%%%%%%%%%%%%%%%%%%%%%%%%%%%%%%%%%%%%%%%%%%%%%%%%%%%%%%%%%%%%%%%%%%%%%%%%%%%%%%%%%%%%%%%%%%%%%%%%%%%%%%%%%%%%%%%
\section{Dissipation of Velocity}

This section devotes to the dissipation estimate for the velocity $\u$.
For this purpose, we consider the linearized system of \eqref{e2} with convective terms
\begin{equation}\label{81}
\begin{cases}
\partial_tb+\u\cdot\nabla b+\Dv\u=\mathcal{K}\\
\partial_t\u+\u\cdot\nabla\u-\D\u-\partial_{x_1}\H+\nabla(b+\H_1)=\mathcal{L}\\
\partial_t\H+\u\cdot\nabla\H+h_0\Dv\u-h_0\cdot\nabla\u=\mathcal{M},
\end{cases}
\end{equation}
where $h_0=(1,0)^\top$, $$\mathcal{K}=-b\Dv\u,$$ 
\begin{equation*}
 \begin{split}
\mathcal{L}&=\left(\f{1}{\r}-1\right)\Big[h_0\cdot\nabla\H+\D\u-\nabla\H_1\Big]+b\nabla b+\f{1}{\r}\left(\H\cdot\nabla\H-\f12\nabla|\H|^2\right),
 \end{split}
\end{equation*}
and
$$\mathcal{M}=\H\cdot\nabla\u-\H\Dv\u.$$
Convective terms are kept in \eqref{81} because no matter how smooth solutions $(b,\u,\H)$ are, convective terms, such as $\u\cdot\nabla b$ and $\u\cdot\nabla\H$, 
always lose one derivative and hence it is hard to treat them as external terms.

We begin with diagonalizing the system \eqref{81}. Set $$\mathfrak{H}=\nabla(b+\H_1)-\partial_{x_1}\H.$$ Equations \eqref{81} give
\begin{equation}\label{81a}
 \begin{split}
\partial_t\mathfrak{H}+Q(D)\u=\nabla(\mathcal{K}+\mathcal{M}_1)-\partial_{x_1}\mathcal{M}-\nabla\Big(\u\cdot\nabla(b+\H_1)\Big)+\partial_{x_1}(\u\cdot\nabla\H).
 \end{split}
\end{equation}
Recall that $\mathcal{P}$ denotes the differential operator with the symbol $\mathcal{P}(\xi)$. Combining \eqref{81} and \eqref{81a} together gives
\begin{subequations}\label{81b}
 \begin{align}
&\partial_t\mathcal{P}\u+\u\cdot\nabla\mathcal{P}\u-\D\mathcal{P}\u+\mathcal{P}\mathcal{H}=\mathcal{O}\label{81ba}\\
&\partial_t\mathcal{H}+\u\cdot\nabla\mathcal{H}-\textrm{diag}(\lambda_{-}(D),\lambda_{+}(D))\mathcal{P}\u=\mathcal{Q}, \label{81bb} 
 \end{align}
\end{subequations}
where
$$\mathcal{O}\overset{def}=-\mathcal{P}(\u\cdot\nabla\u)+\u\cdot\nabla\mathcal{P}\u+\mathcal{P}\mathcal{L}$$ and
$$\mathcal{Q}\overset{def}=\mathcal{P}\Big(\nabla(\mathcal{K}+\mathcal{M}_1)-\partial_{x_1}\mathcal{M}\Big)
+\u\cdot\nabla\mathcal{H}-\mathcal{P}\Big[\nabla\Big(\u\cdot\nabla(b+\H_1)\Big)-\partial_{x_1}(\u\cdot\nabla\H)\Big].$$

Set
$$\mathcal{P}\u=(u,w).$$
Note that Plancherel's identity implies
$$\|\mathcal{P}\u\|_{L^2}=\|\u\|_{L^2}\quad\textrm{and}\quad \|\mathfrak{H}\|_{L^2}=\|\mathcal{H}\|_{L^2}.$$

\subsection{Dissipation of $u$}  
According to \eqref{81b}, $(u,\mathcal{H}_1)$ satisfies
\begin{subequations}\label{e4}
 \begin{align}
&\partial_tu+\u\cdot\nabla u-\D u+\mathcal{H}_1=\mathcal{O}_1\label{e4a}\\
&\partial_t\mathcal{H}_1+\u\cdot\nabla\mathcal{H}_1-\lambda_{-}(D)u=\mathcal{Q}_1.\label{e4b}
 \end{align}
\end{subequations}
Since $\lambda_{-}(\xi)\approx \xi_1^2$, the system \eqref{e4} shows a similar linear sturcture as the incompressible MHD (cf. \cite{HW2}), and this observation suggests us to decompose the frequencies
 into two parts: $k+1\ge 2q$ and $k+1<2q$. Note that the decomposition
of the frequence depends on the competition between the hyperbolicity and the parabolicity.

For the system \eqref{e4}, we have the following dissipation estimate for $u$.
\begin{Proposition}\label{p00}
For solutions $(u, \mathcal{H}_1)$ of \eqref{e4}, it holds true
\begin{equation}\label{82}
\begin{split}
&\|u\|_{L^\infty(\hat{B}^0)}+\|\lambda_{-}(D)^{-\f12}\mathcal{H}_1\|_{L^\infty(\tilde{B}^{0,1})}+\|u\|_{L^1(\hat{B}^2)}\\
&\qquad+\int_0^t\sum_{k+1\ge 2q}2^{2q}\|\lambda_{-}(D)^{-\f12}\hat{\D}_{q,k}\mathcal{H}_1\|_{L^2}ds+\int_0^t\sum_{k+1<2q}\|\hat{\D}_{q,k}\mathcal{H}_1\|_{L^2}ds\\
&\quad\le C\Big(\|u(0)\|_{\hat{B}^0}+\|\lambda_{-}(D)^{-\f12}\mathcal{H}_1(0)\|_{\tilde{B}^{0,1}}+\|\lambda_{-}(D)^{-\f12}\mathcal{Q}_1\|_{L^1(\tilde{B}^{0,1})}+\|\mathcal{O}_1\|_{L^1(\hat{B}^{0})}\\
&\qquad+\|\u\|_{L^1(\hat{B}^2)}(\|\lambda_{-}(D)^{-\f12}\mathcal{H}_1\|_{L^\infty(\tilde{B}^{0,1})}+\|u\|_{L^\infty(\hat{B}^0)})\Big).
\end{split}
\end{equation}
\end{Proposition}
\begin{proof} We split the proof into two steps.

\texttt{Step 1: Energy Estimates.}
Applying the Littlewood-Paley decomposition $\hat{\D}_{q,k}$ to \eqref{e4}, one has
\begin{subequations}\label{83}
 \begin{align}
&\partial_t\hat{\D}_{q,k}u+\hat{\D}_{q,k}(\u\cdot\nabla u)-\D\hat{\D}_{q,k}u+\hat{\D}_{q,k}\mathcal{H}_1=\hat{\D}_{q,k}\mathcal{O}_1\label{83b}\\
&\partial_t\hat{\D}_{q,k}\mathcal{H}_1+\hat{\D}_{q,k}\u\cdot\nabla\mathcal{H}_1-\lambda_{-}(D)\hat{\D}_{q,k}u=\hat{\D}_{q,k}\mathcal{Q}_1.\label{83c}
 \end{align}
\end{subequations}
We consider two cases according to the frequence.

\textit{Case 1: $k+1\ge 2q$.} 

For $\iota>0$, define
\begin{equation*}
\begin{split}
f_{q,k}^2&=\|\hat{\D}_{q,k} u\|_{L^2}^2+\|\lambda_{-}(D)^{-\f12}\hat{\D}_{q,k}\mathcal{H}_1\|_{L^2}^2+\iota 2^{2q-2k+1}(\hat{\D}_{q,k}u|\hat{\D}_{q,k}\mathcal{H}_1).
\end{split}
\end{equation*}

Taking the $L^2$ product of \eqref{83b} with $\hat{\D}_{q,k}u$ gives
\begin{equation}\label{84b}
 \begin{split}
&\f12\f{d}{dt}\|\hat{\D}_{q,k}u\|_{L^2}^2+(\hat{\D}_{q,k}(\u\cdot\nabla u)|\hat{\D}_{q,k}u)+(\hat{\D}_{q,k}\mathcal{H}_1|\hat{\D}_{q,k}u)
+\|\Lambda\hat{\D}_{q,k}u\|_{L^2}^2=(\hat{\D}_{q,k}\mathcal{O}_1|\hat{\D}_{q,k}u).
 \end{split}
\end{equation}
Applying the operator $\lambda_{-}(D)^{-\f12}$ to \eqref{83c} and taking the $L^2$ product of the resulting equation with $\lambda_{-}(D)^{-\f12}\hat{\D}_{q,k}\mathcal{H}_1$ gives
\begin{equation}\label{84c}
 \begin{split}
&\f12\f{d}{dt}\|\lambda_{-}(D)^{-\f12}\hat{\D}_{q,k}\mathcal{H}_1\|_{L^2}^2+(\lambda_{-}(D)^{-\f12}\hat{\D}_{q,k}(\u\cdot\nabla\mathcal{H}_1)|\lambda_{-}(D)^{-\f12}\hat{\D}_{q,k}\mathcal{H}_1)\\
&\quad-(\hat{\D}_{q,k}u|\hat{\D}_{q,k}\mathcal{H}_1)=(\lambda_{-}(D)^{-\f12}\hat{\D}_{q,k}\mathcal{Q}_1|\lambda_{-}(D)^{-\f12}\hat{\D}_{q,k}\mathcal{H}_1).
 \end{split}
\end{equation}

For the cross term $(\hat{\D}_{q,k}u|\hat{\D}_{q,k}\mathcal{H}_1)$, multiplying \eqref{83c} and \eqref{83b} by $\hat{\D}_{q,k}u$ and $\hat{\D}_{q,k}\partial_{x_1}\mathcal{H}_1$ respectively and then summing together, one gets
\begin{equation}\label{85}
 \begin{split}
&\f{d}{dt}(\hat{\D}_{q,k}u|\hat{\D}_{q,k}\mathcal{H}_1)+(\hat{\D}_{q,k}(\u\cdot\nabla u)|\hat{\D}_{q,k}\mathcal{H}_1)+(\hat{\D}_{q,k}(\u\cdot\nabla\mathcal{H}_1)|\hat{\D}_{q,k}u)\\
&\qquad+(\Lambda^2\hat{\D}_{q,k}u|\hat{\D}_{q,k}\mathcal{H}_1)+\|\hat{\D}_{q,k}\mathcal{H}_1\|_{L^2}^2-(\lambda_{-}(D)\hat{\D}_{q,k}u|\hat{\D}_{q,k}u)\\
&\quad=(\hat{\D}_{q,k}\mathcal{O}_1|\hat{\D}_{q,k}\mathcal{H}_1)+(\hat{\D}_{q,k}\mathcal{Q}_1|\hat{\D}_{q,k}u).
 \end{split}
\end{equation}

A linear combination of \eqref{84b}-\eqref{85} gives
\begin{equation}\label{87}
 \begin{split}
&\f12\f{d}{dt}f_q^2+(1-\iota)\|\Lambda\hat{\D}_{q,k}u\|_{L^2}^2+\iota\|\Lambda\lambda_{-}(D)^{-\f12}\hat{\D}_{q,k} \mathcal{H}_1\|_{L^2}^2+\iota2^{2q-2k}(\Lambda^2\hat{\D}_{q,k}u|\hat{\D}_{q,k}\mathcal{H}_1)\\
&\quad\approx \mathcal{X}_{q,k}
 \end{split}
\end{equation}
with
\begin{equation*}
 \begin{split}
\mathcal{X}_{q,k}&=(\hat{\D}_{q,k}\mathcal{O}_1|\hat{\D}_{q,k}u)+(\lambda_{-}(D)^{-\f12}\hat{\D}_{q,k}\mathcal{Q}_1|\lambda_{-}(D)^{-\f12}\hat{\D}_{q,k}\mathcal{H}_1)-(\hat{\D}_{q,k}(\u\cdot\nabla u)|\hat{\D}_{q,k}u)\\
&\quad-(\lambda_{-}(D)^{-\f12}\hat{\D}_{q,k}(\u\cdot\nabla\mathcal{H}_1)|\lambda_{-}(D)^{-\f12}\hat{\D}_{q,k}\mathcal{H}_1)+\iota2^{2q-2k}\Big((\hat{\D}_{q,k}\mathcal{O}_1|\hat{\D}_{q,k}\mathcal{H}_1)\\
&\quad+(\hat{\D}_{q,k}\mathcal{Q}_1|\hat{\D}_{q,k}u)-(\hat{\D}_{q,k}(\u\cdot\nabla u)|\hat{\D}_{q,k}\mathcal{H}_1)-(\hat{\D}_{q,k}(\u\cdot\nabla\mathcal{H}_1)|\hat{\D}_{q,k}u)\Big).
 \end{split}
\end{equation*}

Note that
$$f_{q,k}^2\approx \|\hat{\D}_{q,k}u\|_{L^2}^2+\|\lambda_{-}(D)^{-\f12}\hat{\D}_{q,k}\mathcal{H}_1\|_{L^2}^2$$if $\iota$ is chosen to be sufficiently small.

Since $k+1\ge 2q$, $q\lesssim 1$. Using Bernstein's inequality and the Cauchy-Swartz inequality, one has 
\begin{equation*}
 \begin{split}
&(1-\iota)\|\Lambda\hat{\D}_{q,k}u\|_{L^2}^2+\iota\|\Lambda\lambda_{-}(D)^{-\f12}\hat{\D}_{q,k} \mathcal{H}_1\|_{L^2}^2+\iota2^{2q-2k}(\Lambda^2\hat{\D}_{q,k}u|\hat{\D}_{q,k}\mathcal{H}_1)\\
&\quad\approx \|\Lambda\hat{\D}_{q,k} u\|_{L^2}^2+\|\Lambda\lambda_{-}(D)^{-\f12}\hat{\D}_{q,k} \mathcal{H}_1\|_{L^2}^2\\
&\quad\ge C2^{2q}\Big(\|\hat{\D}_{q,k} u\|_{L^2}^2+\|\lambda_{-}(D)^{-\f12}\hat{\D}_{q,k} \mathcal{H}_1\|_{L^2}^2\Big).
 \end{split}
\end{equation*}

Since $k+1\ge 2q$, one deduces from Lemma \ref{cl} that
\begin{equation*}
\begin{split}
|\mathcal{X}_{q,k}|&\le Cf_{q,k}\Big(\|\hat{\D}_{q,k}\mathcal{O}_1\|_{L^2}+\|\lambda_{-}(D)^{-\f12}\hat{\D}_{q,k}\mathcal{Q}_1\|_{L^2}\\
&\quad+\alpha_{q,k}\|\u\|_{\hat{B}^2}(\|\lambda_{-}(D)^{-\f12}\mathcal{H}_1\|_{\tilde{B}^{0,1}}+\|u\|_{\hat{B}^0})\Big),
\end{split}
\end{equation*} 
and hence it follows from \eqref{87} that there is a positive constant $\kappa$ such that
\begin{equation}\label{89}
\begin{split}
\f{d}{dt}f_{q,k}+\kappa 2^{2q} f_{q,k}&\le Cf_{q,k}\Big(\|\hat{\D}_{q,k}\mathcal{O}_1\|_{L^2}+\|\lambda_{-}(D)^{-\f12}\hat{\D}_{q,k}\mathcal{Q}_1\|_{L^2}\\
&\quad+\alpha_{q,k}\|\u\|_{\hat{B}^2}(\|\lambda_{-}(D)^{-\f12}\mathcal{H}_1\|_{\tilde{B}^{0,1}}+\|u\|_{\hat{B}^0})\Big).
\end{split}
\end{equation}
Summing \eqref{89} over $q, k\in\mathbb{Z}$ and integrating over $t\in \R$, one obtains \eqref{82}.

\textit{Case 2: $k+1<2q$.} 

We define
\begin{equation*}
 \begin{split}
f_{q, k}^2&=2\|\Lambda^{-2}\lambda_{-}(D)\hat{\D}_{q,k}u\|_{L^2}^2+\|\hat{\D}_{q,k}\mathcal{H}_1\|_{L^2}^2+2(\Lambda^{-2}\lambda_{-}(D)\hat{\D}_{q,k}u|\hat{\D}_{q,k}\mathcal{H}_1).
 \end{split}
\end{equation*}

Applying the operator $\Lambda^{-2}\lambda_{-}(D)$ to \eqref{83b}, and then taking the $L^2$ product of the resulting equation with $\Lambda^{-2}\lambda_{-}(D)\hat{\D}_{q,k}u$, one gets
\begin{equation}\label{810a}
 \begin{split}
&\f12\f{d}{dt}\|\Lambda^{-2}\lambda_{-}(D)\hat{\D}_{q,k}u\|_{L^2}^2+(\Lambda^{-2}\lambda_{-}(D)\hat{\D}_{q,k}(\u\cdot\nabla u)|\Lambda^{-2}\lambda_{-}(D)\hat{\D}_{q,k}u)\\
&\qquad+\|\Lambda\Lambda^{-2}\lambda_{-}(D)\hat{\D}_{q,k}u\|_{L^2}^2+(\Lambda^{-2}\lambda_{-}(D)\hat{\D}_{q,k}\mathcal{H}_1|\Lambda^{-2}\lambda_{-}(D)\hat{\D}_{q,k}u)\\
&\quad=(\Lambda^{-2}\lambda_{-}(D)\hat{\D}_{q,k}\mathcal{O}_1|\Lambda^{-2}\lambda_{-}(D)\hat{\D}_{q,k}u).  
 \end{split}
\end{equation}
Taking the $L^2$ product of \eqref{83c} with $\hat{\D}_{q,k}\mathcal{H}_1$, one gets
\begin{equation}\label{810b}
\begin{split}
 &\f12\f{d}{dt}\|\hat{\D}_{q,k}\mathcal{H}_1\|_{L^2}^2+(\hat{\D}_{q,k}(\u\cdot\nabla \mathcal{H}_1))|\hat{\D}_{q,k}\mathcal{H}_1)
-(\lambda_{-}(D)\hat{\D}_{q,k}u|\hat{\D}_{q,k}\mathcal{H}_1)=(\hat{\D}_{q,k}\mathcal{Q}_1|\hat{\D}_{q,k}\mathcal{H}_1).
\end{split}
\end{equation}
For the cross term $(\Lambda^{-2}\lambda_{-}(D)\hat{\D}_{q,k}u|\hat{\D}_{q,k}\mathcal{H}_1)$, applying the operator $\Lambda^{-2}\lambda_{-}(D)$ to \eqref{83b}, 
and then taking the $L^2$ product of the resulting equation and \eqref{83c} with $\hat{\D}_{q,k}\mathcal{H}_1$ and $\Lambda^{-2}\lambda_{-}(D)\hat{\D}_{q,k}u$ respectively, one obtains
\begin{equation}\label{810c}
 \begin{split}
&\f{d}{dt}(\Lambda^{-2}\lambda_{-}(D)\hat{\D}_{q,k}u|\hat{\D}_{q,k}\mathcal{H}_1)+(\Lambda^{-2}\lambda_{-}(D)\hat{\D}_{q,k}(\u\cdot\nabla u)|\hat{\D}_{q,k}\mathcal{H}_1)\\
&\qquad+(\hat{\D}_{q,k}(\u\cdot\nabla\mathcal{H}_1)|\Lambda^{-2}\lambda_{-}(D)\hat{\D}_{q,k}u)+(\lambda_{-}(D)\hat{\D}_{q,k}u|\hat{\D}_{q,k}\mathcal{H}_1)\\
&\qquad+\|\Lambda^{-1}\lambda_{-}(D)^{\f12}\hat{\D}_{q,k}\mathcal{H}_1\|_{L^2}^2
-\|\Lambda^{-1}\lambda_{-}(D)\hat{\D}_{q,k}u\|_{L^2}^2\\
&\quad=(\Lambda^{-2}\lambda_{-}(D)\hat{\D}_{q,k}\mathcal{O}_1|\hat{\D}_{q,k}\mathcal{H}_1)+(\hat{\D}_{q,k}\mathcal{Q}_1|\Lambda^{-2}\lambda_{-}(D)\hat{\D}_{q,k}u),
 \end{split}
\end{equation}

A linear combination of \eqref{810a}, \eqref{810b} and \eqref{810c} gives
\begin{equation}\label{811}
\begin{split}
&\f12\f{d}{dt}f_q^2+\|\Lambda\Lambda^{-2}\lambda_{-}(D)\hat{\D}_{q,k}u\|_{L^2}^2+\|\Lambda^{-1}\lambda_{-}(D)^{\f12}\hat{\D}_{q,k}\mathcal{H}_1\|_{L^2}^2
\\
&\quad+2(\Lambda^{-2}\lambda_{-}(D)\hat{\D}_{q,k}\mathcal{H}_1|\Lambda^{-2}\lambda_{-}(D)\hat{\D}_{q,k}u)\\
&=\mathcal{Y}_{q,k}
\end{split}
\end{equation}
with
\begin{equation*}
 \begin{split}
\mathcal{Y}_{q,k}&\overset{def}=2(\Lambda^{-2}\lambda_{-}(D)\hat{\D}_{q,k}\mathcal{O}_1|\Lambda^{-2}\lambda_{-}(D)\hat{\D}_{q,k}u)+(\hat{\D}_{q,k}\mathcal{Q}_1|\hat{\D}_{q,k}\mathcal{H}_1)\\
&\quad+2(\Lambda^{-2}\lambda_{-}(D)\hat{\D}_{q,k}(\u\cdot\nabla u)|\Lambda^{-2}\lambda_{-}(D)\hat{\D}_{q,k}u)+(\hat{\D}_{q,k}(\u\cdot\nabla \mathcal{H}_1))|\hat{\D}_{q,k}\mathcal{H}_1)\\
&\quad+(\Lambda^{-2}\lambda_{-}(D)\hat{\D}_{q,k}\mathcal{O}_1|\hat{\D}_{q,k}\mathcal{H}_1)+(\hat{\D}_{q,k}\mathcal{Q}_1|\Lambda^{-2}\lambda_{-}(D)\hat{\D}_{q,k}u)\\
&\quad-(\Lambda^{-2}\lambda_{-}(D)\hat{\D}_{q,k}(\u\cdot\nabla u)|\hat{\D}_{q,k}\mathcal{H}_1)-(\hat{\D}_{q,k}(\u\cdot\nabla\mathcal{H}_1)|\Lambda^{-2}\lambda_{-}(D)\hat{\D}_{q,k}u).
 \end{split}
\end{equation*}

Note that 
\begin{equation}\label{a12a}
f_{q,k}^2\approx \|\Lambda^{-2}\lambda_{-}(D)\hat{\D}_{q,k}u\|_{L^2}^2+\|\hat{\D}_{q,k}\mathcal{H}_1\|_{L^2}^2.
\end{equation} 
Since $k+1<2q$, it holds
\begin{equation*}
 \begin{split}
&2|(\Lambda^{-2}\lambda_{-}(D)\hat{\D}_{q,k}\mathcal{H}_1|\Lambda^{-2}\lambda_{-}(D)\hat{\D}_{q,k}u)|\\
&\quad\le \f12\Big(
\|\Lambda\Lambda^{-2}\lambda_{-}(D)\hat{\D}_{q,k}u\|_{L^2}^2+\|\Lambda^{-1}\lambda_{-}(D)^{\f12}\hat{\D}_{q,k}\mathcal{H}_1\|_{L^2}^2\Big), 
 \end{split}
\end{equation*}
and hence
\begin{equation*}
 \begin{split}
&\|\Lambda\Lambda^{-2}\lambda_{-}(D)\hat{\D}_{q,k}u\|_{L^2}^2+\|\Lambda^{-1}\lambda_{-}(D)^{\f12}\hat{\D}_{q,k}\mathcal{H}_1\|_{L^2}^2
+2(\Lambda^{-2}\lambda_{-}(D)\hat{\D}_{q,k}\mathcal{H}_1|\Lambda^{-2}\lambda_{-}(D)\hat{\D}_{q,k}u)\\
&\quad\approx \|\Lambda\Lambda^{-2}\lambda_{-}(D)\hat{\D}_{q,k}u\|_{L^2}^2+\|\Lambda^{-1}\lambda_{-}(D)^{\f12}\hat{\D}_{q,k}\mathcal{H}_1\|_{L^2}^2\\
&\quad\ge C2^{2k-2q}\Big(\|\Lambda^{-2}\lambda_{-}(D)\hat{\D}_{q,k}u\|_{L^2}^2+\|\hat{\D}_{q,k}\mathcal{H}_1\|_{L^2}^2\Big).  
 \end{split}
\end{equation*}

On the other hand, one deduces from Lemma \ref{cl} that
\begin{equation}\label{a12}
\begin{split}
&|\mathcal{Y}_{q,k}|\\
&\le Cf_{q,k}\Big(\|\hat{\D}_{q,k}\mathcal{O}_1\|_{L^2}+\|\hat{\D}_{q,k}\mathcal{Q}_1\|_{L^2}+\alpha_{q,k}2^{2k-2q}\|\u\|_{\hat{B}^2}(\|\lambda_{-}(D)^{-\f12}\mathcal{H}_1\|_{\tilde{B}^{0,1}}+\|u\|_{\hat{B}^0})\Big).
\end{split}
\end{equation}

Combining \eqref{811} and \eqref{a12}, one has
\begin{equation}\label{a13}
 \begin{split}
\f{d}{dt}f_{q,k}+\kappa 2^{2k-2q} f_{q,k}&\lesssim \|\hat{\D}_{q,k}\mathcal{O}_1\|_{L^2}+\|\hat{\D}_{q,k}\mathcal{Q}_1\|_{L^2}+\alpha_{q,k}2^{2q-2k}\|\u\|_{\hat{B}^2}(\|\mathcal{H}_1\|_{\tilde{B}^{0,1}}+\|u\|_{\hat{B}^0})
 \end{split}
\end{equation} for some positive constant $\kappa$.

Multiplying \eqref{a13} by $2^{2q-2k}$, summing the resulting equation over $q,k\in\mathbb{Z}$ and integrating over $t\in \R$, it follows
\begin{equation}\label{a15}
\begin{split}
&\sup_{0\le \tau<\infty}\sum_{k+1< 2q}\Big(\|\hat{\D}_{q,k}u(\tau)\|_{L^2}+2^{2q-k}\|\lambda_{-}(D)^{-\f12}\hat{\D}_{q,k}\mathcal{H}_1(\tau)\|_{L^2}\Big)+\sum_{k+1<2q}\|\hat{\D}_{q,k}\mathcal{H}_1\|_{L^1(L^2)}\\
&\quad\le C\Big(\|u(0)\|_{\hat{B}^0}+\|\lambda_{-}(D)^{-\f12}\mathcal{H}_1(0)\|_{\tilde{B}^{0,1}}+\|\mathcal{O}_1\|_{L^1(\hat{B}^0)}+\|\lambda_{-}(D)^{-\f12}\mathcal{Q}_1\|_{L^1(\tilde{B}^{0,1})}\\
&\qquad+\|\u\|_{L^1(\hat{B}^2)}(\|u\|_{L^\infty(\hat{B}^0)}+\|\lambda_{-}(D)^{-\f12}\mathcal{H}_1\|_{L^\infty(\tilde{B}^{0,1})})\Big).
\end{split}
\end{equation}

\texttt{Step 2: Smoothing Effect.} We are going to use the parabolicity of $u$ to improve the dissipation of $u$. According to \eqref{89}, it is only left to handle the case $k+1<2q$.
Going back to the equation \eqref{83b}, and regarding now the term $\mathcal{H}_1$ as an external term, one has
\begin{equation*}
\begin{split}
&\f12\f{d}{dt}\|\hat{\D}_{q,k}u\|_{L^2}^2+(\hat{\D}_{q,k}(\u\cdot\nabla u)|\hat{\D}_{q,k}u)+2^{2q}\|\hat{\D}_{q,k}u\|_{L^2}\\
&\quad\le C\Big(\|\hat{\D}_{q,k}\mathcal{H}_1\|_{L^2}+\|\hat{\D}_{q,k}\mathcal{O}_1\|_{L^2}\Big)\|\hat{\D}_{q,k}u\|_{L^2}, 
\end{split}
\end{equation*}
and hence
\begin{equation}\label{a16}
\begin{split}
\sum_{k+1<2q}2^{2q}\|\hat{\D}_{q,k}u\|_{L^1(L^2)}&\le C\Big(\|u(0)\|_{\hat{B}^0}+\|\mathcal{L}_1\|_{L^1(\hat{B}^0)}+\|\u\|_{\hat{B}^2}\|u\|_{\hat{B}^0}\\
&\quad+\sum_{k+1<2q}\|\hat{\D}_{q,k}\mathcal{H}_1\|_{L^1(L^2)}\Big).
\end{split}
\end{equation}
The desired estimate \eqref{82} follows from \eqref{89}, \eqref{a15} and \eqref{a16}.

\end{proof}

\subsection{Dissipation of $w$}

According to \eqref{81b}, $(w,\mathcal{H}_2)$ satisfies
\begin{subequations}\label{e41}
 \begin{align}
&\partial_tw+\u\cdot\nabla w-\D w+\mathcal{H}_2=\mathcal{O}_2\label{e4a1}\\
&\partial_t\mathcal{H}_2+\u\cdot\nabla\mathcal{H}_2-\lambda_{+}(D)w=\mathcal{Q}_2.\label{e4b1}
 \end{align}
\end{subequations}
Since $\lambda_{+}(\xi)\approx |\xi|^2$, the system \eqref{e41} features a similar linear structure as compressible Navier-Stokes equations and viscoelastic 
systems (cf. \cite{RD, HW, QZ}), and this suggests us to decompose the frequence 
into two parts: low frequencies $q\le 1$ and high frequencies $q>1$.

For the system \eqref{e41}, we have the following dissipation estimate for $w$.
\begin{Proposition}\label{p111}
For solutions $(w, \mathcal{H}_2)$ of \eqref{e41}, it holds true
\begin{equation}\label{821}
\begin{split}
&\|w\|_{L^\infty(\hat{B}^0)}+\|\lambda_{+}(D)^{-\f12}\mathcal{H}_2\|_{L^\infty(\hat{B}^{0}\cap\hat{B}^1)}+\|w\|_{L^1(\hat{B}^2)}\\
&\qquad+\int_0^t\sum_{q\le 1}2^{2q}\|\lambda_{+}(D)^{-\f12}\hat{\D}_{q,k}\mathcal{H}_2\|_{L^2}ds+\int_0^t\sum_{q> 1}\|\hat{\D}_{q,k}\mathcal{H}_2\|_{L^2}ds\\
&\quad\le C\Big(\|w(0)\|_{\hat{B}^0}+\|\lambda_{+}(D)^{-\f12}\mathcal{H}_2(0)\|_{\hat{B}^{0}\cap\hat{B}^1}+\|\lambda_{+}(D)^{-\f12}\mathcal{Q}_2\|_{L^1(\hat{B}^{0}\cap\hat{B}^1)}+\|\mathcal{O}_2\|_{L^1(\hat{B}^{0})}\\
&\qquad+\|\u\|_{L^1(\hat{B}^2)}(\|\lambda_{+}(D)^{-\f12}\mathcal{H}_2\|_{L^\infty(\hat{B}^{0}\cap\hat{B}^1)}+\|w\|_{L^\infty(\hat{B}^0)})\Big).
\end{split}
\end{equation}
\end{Proposition}
\begin{proof} We split the proof into two steps.

\texttt{Step 1: Energy Estimates.}
Applying the Littlewood-Paley decomposition $\hat{\D}_{q,k}$ to \eqref{e41}, one has
\begin{subequations}\label{831}
 \begin{align}
&\partial_t\hat{\D}_{q,k}w+\hat{\D}_{q,k}(\u\cdot\nabla w)-\D\hat{\D}_{q,k}w+\hat{\D}_{q,k}\mathcal{H}_2=\hat{\D}_{q,k}\mathcal{O}_2\label{831b}\\
&\partial_t\hat{\D}_{q,k}\mathcal{H}_2+\hat{\D}_{q,k}(\u\cdot\nabla\mathcal{H}_2)-\lambda_{+}(D)\hat{\D}_{q,k}w=\hat{\D}_{q,k}\mathcal{Q}_2.\label{831c}
 \end{align}
\end{subequations}
We consider two cases according to the frequence.

\textit{Case 1: $q\le 1$.} 

For $\iota>0$, define
\begin{equation*}
\begin{split}
f_{q,k}^2&=\|\hat{\D}_{q,k} w\|_{L^2}^2+\|\lambda_{+}(D)^{-\f12}\hat{\D}_{q,k}\mathcal{H}_2\|_{L^2}^2+2\iota (\hat{\D}_{q,k}w|\hat{\D}_{q,k}\mathcal{H}_2).
\end{split}
\end{equation*}

Taking the $L^2$ product of \eqref{831b} with $\hat{\D}_{q,k}u$ gives
\begin{equation}\label{841b}
 \begin{split}
&\f12\f{d}{dt}\|\hat{\D}_{q,k}w\|_{L^2}^2+(\hat{\D}_{q,k}(\u\cdot\nabla w)|\hat{\D}_{q,k}w)+(\hat{\D}_{q,k}\mathcal{H}_2|\hat{\D}_{q,k}w)
+\|\Lambda\hat{\D}_{q,k}w\|_{L^2}^2\\
&\quad=(\hat{\D}_{q,k}\mathcal{O}_2|\hat{\D}_{q,k}w).
 \end{split}
\end{equation}
Applying the operator $\lambda_{+}(D)^{-\f12}$ to \eqref{831c} and taking the $L^2$ product of the resulting equation with $\lambda_{+}(D)^{-\f12}\hat{\D}_{q,k}\mathcal{H}_2$ gives
\begin{equation}\label{841c}
 \begin{split}
&\f12\f{d}{dt}\|\lambda_{+}(D)^{-\f12}\hat{\D}_{q,k}\mathcal{H}_2\|_{L^2}^2+(\lambda_{+}(D)^{-\f12}\hat{\D}_{q,k}(\u\cdot\nabla\mathcal{H}_2)|\lambda_{+}(D)^{-\f12}\hat{\D}_{q,k}\mathcal{H}_2)\\
&\quad-(\hat{\D}_{q,k}w|\hat{\D}_{q,k}\mathcal{H}_2)=(\lambda_{+}(D)^{-\f12}\hat{\D}_{q,k}\mathcal{Q}_2|\lambda_{+}(D)^{-\f12}\hat{\D}_{q,k}\mathcal{H}_2).
 \end{split}
\end{equation}

For the cross term $(\hat{\D}_{q,k}w|\hat{\D}_{q,k}\mathcal{H}_2)$, multiplying \eqref{831c} and \eqref{831b} by $\hat{\D}_{q,k}w$ and $\hat{\D}_{q,k}\mathcal{H}_2$ respectively and then summing together, one gets
\begin{equation}\label{851}
 \begin{split}
&\f{d}{dt}(\hat{\D}_{q,k}w|\hat{\D}_{q,k}\mathcal{H}_2)+(\hat{\D}_{q,k}(\u\cdot\nabla w)|\hat{\D}_{q,k}\mathcal{H}_2)+(\hat{\D}_{q,k}(\u\cdot\nabla\mathcal{H}_2)|\hat{\D}_{q,k}w)\\
&\qquad+(\Lambda^2\hat{\D}_{q,k}w|\hat{\D}_{q,k}\mathcal{H}_2)+\|\hat{\D}_{q,k}\mathcal{H}_2\|_{L^2}^2-(\lambda_{+}(D)\hat{\D}_{q,k}w|\hat{\D}_{q,k}w)\\
&\quad=(\hat{\D}_{q,k}\mathcal{O}_2|\hat{\D}_{q,k}\mathcal{H}_2)+(\hat{\D}_{q,k}\mathcal{Q}_2|\hat{\D}_{q,k}w).
 \end{split}
\end{equation}

A linear combination of \eqref{841b}-\eqref{851} gives
\begin{equation}\label{871}
 \begin{split}
&\f12\f{d}{dt}f_q^2+(1-\iota)\|\Lambda\hat{\D}_{q,k}w\|_{L^2}^2+\iota\|\Lambda\lambda_{+}(D)^{-\f12}\hat{\D}_{q,k} \mathcal{H}_2\|_{L^2}^2+\iota(\Lambda^2\hat{\D}_{q,k}w|\hat{\D}_{q,k}\mathcal{H}_2)\\
&\quad\approx \mathcal{X}_{q,k}
 \end{split}
\end{equation}
with
\begin{equation*}
 \begin{split}
\mathcal{X}_{q,k}&=(\hat{\D}_{q,k}\mathcal{O}_2|\hat{\D}_{q,k}w)+(\lambda_{+}(D)^{-\f12}\hat{\D}_{q,k}\mathcal{Q}_2|\lambda_{+}(D)^{-\f12}\hat{\D}_{q,k}\mathcal{H}_2)-(\hat{\D}_{q,k}(\u\cdot\nabla w)|\hat{\D}_{q,k}w)\\
&\quad-(\lambda_{+}(D)^{-\f12}\hat{\D}_{q,k}(\u\cdot\nabla\mathcal{H}_2)|\lambda_{+}(D)^{-\f12}\hat{\D}_{q,k}\mathcal{H}_2)+\iota\Big((\hat{\D}_{q,k}\mathcal{O}_2|\hat{\D}_{q,k}\mathcal{H}_2)\\
&\quad+(\hat{\D}_{q,k}\mathcal{Q}_2|\hat{\D}_{q,k}w)-(\hat{\D}_{q,k}(\u\cdot\nabla w)|\hat{\D}_{q,k}\mathcal{H}_2)-(\hat{\D}_{q,k}(\u\cdot\nabla\mathcal{H}_2)|\hat{\D}_{q,k}w)\Big).
 \end{split}
\end{equation*}

Note that
$$f_{q,k}^2\approx \|\hat{\D}_{q,k}w\|_{L^2}^2+\|\lambda_{+}(D)^{-\f12}\hat{\D}_{q,k}\mathcal{H}_2\|_{L^2}^2$$if $\iota$ is chosen to be sufficiently small.

Using Bernstein's inequality and the Cauchy-Swartz inequality, one has 
\begin{equation*}
 \begin{split}
&(1-\iota)\|\Lambda\hat{\D}_{q,k}w\|_{L^2}^2+\iota\|\Lambda\lambda_{+}(D)^{-\f12}\hat{\D}_{q,k} \mathcal{H}_2\|_{L^2}^2+\iota(\Lambda^2\hat{\D}_{q,k}w|\hat{\D}_{q,k}\mathcal{H}_2)\\
&\quad\approx \|\Lambda\hat{\D}_{q,k} w\|_{L^2}^2+\|\Lambda\lambda_{+}(D)^{-\f12}\hat{\D}_{q,k} \mathcal{H}_2\|_{L^2}^2\\
&\quad\ge C2^{2q}\Big(\|\hat{\D}_{q,k} w\|_{L^2}^2+\|\lambda_{+}(D)^{-\f12}\hat{\D}_{q,k} \mathcal{H}_2\|_{L^2}^2\Big).
 \end{split}
\end{equation*}

One deduces from Lemma 5.1 in \cite{RD} that
\begin{equation*}
\begin{split}
|\mathcal{X}_{q,k}|&\le Cf_{q,k}\Big(\|\hat{\D}_{q,k}\mathcal{O}_2\|_{L^2}+\|\lambda_{+}(D)^{-\f12}\hat{\D}_{q,k}\mathcal{Q}_2\|_{L^2}\\
&\quad+\alpha_{q,k}\|\u\|_{\hat{B}^2}(\|\lambda_{+}(D)^{-\f12}\mathcal{H}_2\|_{\hat{B}^{0}\cap\hat{B}^1}+\|w\|_{\hat{B}^0})\Big),
\end{split}
\end{equation*} 
and hence it follows from \eqref{871} that there is a positive constant $\kappa$ such that
\begin{equation}\label{891}
\begin{split}
\f{d}{dt}f_{q,k}+\kappa 2^{2q} f_{q,k}&\le Cf_{q,k}\Big(\|\hat{\D}_{q,k}\mathcal{O}_2\|_{L^2}+\|\lambda_{+}(D)^{-\f12}\hat{\D}_{q,k}\mathcal{Q}_2\|_{L^2}\\
&\quad+\alpha_{q,k}\|\u\|_{\hat{B}^2}(\|\lambda_{+}(D)^{-\f12}\mathcal{H}_2\|_{\hat{B}^{0}\cap\hat{B}^1}+\|w\|_{\hat{B}^0})\Big).
\end{split}
\end{equation}
Summing \eqref{891} over $q, k\in\mathbb{Z}$ and integrating over $t\in \R$, one obtains \eqref{821}.

\textit{Case 2: $q>1$.} 

We define
\begin{equation*}
 \begin{split}
f_{q, k}^2&=2\|\Lambda^{-2}\lambda_{+}(D)\hat{\D}_{q,k}w\|_{L^2}^2+\|\hat{\D}_{q,k}\mathcal{H}_2\|_{L^2}^2+2(\Lambda^{-2}\lambda_{+}(D)\hat{\D}_{q,k}w|\hat{\D}_{q,k}\mathcal{H}_2).
 \end{split}
\end{equation*}

Applying the operator $\Lambda^{-2}\lambda_{+}(D)$ to \eqref{831b}, and then taking the $L^2$ product of the resulting equation with $\Lambda^{-2}\lambda_{+}(D)\hat{\D}_{q,k}w$, one gets
\begin{equation}\label{810a1}
 \begin{split}
&\f12\f{d}{dt}\|\Lambda^{-2}\lambda_{+}(D)\hat{\D}_{q,k}w\|_{L^2}^2+(\Lambda^{-2}\lambda_{+}(D)\hat{\D}_{q,k}(\u\cdot\nabla w)|\Lambda^{-2}\lambda_{+}(D)\hat{\D}_{q,k}w)\\
&\qquad+\|\Lambda\Lambda^{-2}\lambda_{+}(D)\hat{\D}_{q,k}w\|_{L^2}^2+(\Lambda^{-2}\lambda_{+}(D)\hat{\D}_{q,k}\mathcal{H}_2|\Lambda^{-2}\lambda_{+}(D)\hat{\D}_{q,k}w)\\
&\quad=(\Lambda^{-2}\lambda_{+}(D)\hat{\D}_{q,k}\mathcal{O}_2|\Lambda^{-2}\lambda_{+}(D)\hat{\D}_{q,k}w).  
 \end{split}
\end{equation}
Taking the $L^2$ product of \eqref{831c} with $\hat{\D}_{q,k}\mathcal{H}_2$, one gets
\begin{equation}\label{810b1}
\begin{split}
 &\f12\f{d}{dt}\|\hat{\D}_{q,k}\mathcal{H}_2\|_{L^2}^2+(\hat{\D}_{q,k}(\u\cdot\nabla \mathcal{H}_2))|\hat{\D}_{q,k}\mathcal{H}_2)
-(\lambda_{+}(D)\hat{\D}_{q,k}w|\hat{\D}_{q,k}\mathcal{H}_2)=(\hat{\D}_{q,k}\mathcal{Q}_2|\hat{\D}_{q,k}\mathcal{H}_2).
\end{split}
\end{equation}
For the cross term $(\Lambda^{-2}\lambda_{+}(D)\hat{\D}_{q,k}w|\hat{\D}_{q,k}\mathcal{H}_2)$, applying the operator $\Lambda^{-2}\lambda_{+}(D)$ to \eqref{831b}, 
and then taking the $L^2$ product of the resulting equation and \eqref{831c} with $\hat{\D}_{q,k}\mathcal{H}_2$ and $\Lambda^{-2}\lambda_{+}(D)\hat{\D}_{q,k}w$ respectively, one obtains
\begin{equation}\label{810c1}
 \begin{split}
&\f{d}{dt}(\Lambda^{-2}\lambda_{+}(D)\hat{\D}_{q,k}w|\hat{\D}_{q,k}\mathcal{H}_2)+(\Lambda^{-2}\lambda_{+}(D)\hat{\D}_{q,k}(\u\cdot\nabla w)|\hat{\D}_{q,k}\mathcal{H}_2)\\
&\qquad+(\hat{\D}_{q,k}(\u\cdot\nabla\mathcal{H}_2)|\Lambda^{-2}\lambda_{+}(D)\hat{\D}_{q,k}w)+(\lambda_{+}(D)\hat{\D}_{q,k}w|\hat{\D}_{q,k}\mathcal{H}_2)\\
&\qquad+\|\Lambda^{-1}\lambda_{+}(D)^{\f12}\hat{\D}_{q,k}\mathcal{H}_2\|_{L^2}^2
-\|\Lambda^{-1}\lambda_{+}(D)\hat{\D}_{q,k}w\|_{L^2}^2\\
&\quad=(\Lambda^{-2}\lambda_{+}(D)\hat{\D}_{q,k}\mathcal{O}_2|\hat{\D}_{q,k}\mathcal{H}_2)+(\hat{\D}_{q,k}\mathcal{Q}_2|\Lambda^{-2}\lambda_{+}(D)\hat{\D}_{q,k}w).
 \end{split}
\end{equation}

A linear combination of \eqref{810a1}, \eqref{810b1} and \eqref{810c1} gives
\begin{equation}\label{8111}
\begin{split}
&\f12\f{d}{dt}f_q^2+\|\Lambda\Lambda^{-2}\lambda_{+}(D)\hat{\D}_{q,k}w\|_{L^2}^2+\|\Lambda^{-1}\lambda_{+}(D)^{\f12}\hat{\D}_{q,k}\mathcal{H}_2\|_{L^2}^2
\\
&\quad+2(\Lambda^{-2}\lambda_{+}(D)\hat{\D}_{q,k}\mathcal{H}_2|\Lambda^{-2}\lambda_{+}(D)\hat{\D}_{q,k}w)\\
&=\mathcal{Y}_{q,k}
\end{split}
\end{equation}
with
\begin{equation*}
 \begin{split}
\mathcal{Y}_{q,k}&\overset{def}=2(\Lambda^{-2}\lambda_{+}(D)\hat{\D}_{q,k}\mathcal{O}_2|\Lambda^{-2}\lambda_{+}(D)\hat{\D}_{q,k}w)+(\hat{\D}_{q,k}\mathcal{Q}_2|\hat{\D}_{q,k}\mathcal{H}_2)\\
&\quad-2(\Lambda^{-2}\lambda_{+}(D)\hat{\D}_{q,k}(\u\cdot\nabla w)|\Lambda^{-2}\lambda_{+}(D)\hat{\D}_{q,k}w)+(\hat{\D}_{q,k}(\u\cdot\nabla \mathcal{H}_2))|\hat{\D}_{q,k}\mathcal{H}_2)\\
&\quad+(\Lambda^{-2}\lambda_{+}(D)\hat{\D}_{q,k}\mathcal{O}_2|\hat{\D}_{q,k}\mathcal{H}_2)+(\hat{\D}_{q,k}\mathcal{Q}_2|\Lambda^{-2}\lambda_{+}(D)\hat{\D}_{q,k}w)\\
&\quad-(\Lambda^{-2}\lambda_{+}(D)\hat{\D}_{q,k}(\u\cdot\nabla w)|\hat{\D}_{q,k}\mathcal{H}_2)-(\hat{\D}_{q,k}(\u\cdot\nabla\mathcal{H}_2)|\Lambda^{-2}\lambda_{+}(D)\hat{\D}_{q,k}w).
 \end{split}
\end{equation*}

Note that 
\begin{equation}\label{a12a1}
f_{q,k}^2\approx \|\Lambda^{-2}\lambda_{+}(D)\hat{\D}_{q,k}w\|_{L^2}^2+\|\hat{\D}_{q,k}\mathcal{H}_2\|_{L^2}^2.
\end{equation} 
Since $q>1$, it holds
\begin{equation*}
 \begin{split}
&2|(\Lambda^{-2}\lambda_{+}(D)\hat{\D}_{q,k}\mathcal{H}_2|\Lambda^{-2}\lambda_{+}(D)\hat{\D}_{q,k}w)|\\
&\quad\le \f12\Big(
\|\Lambda\Lambda^{-2}\lambda_{+}(D)\hat{\D}_{q,k}w\|_{L^2}^2+\|\Lambda^{-1}\lambda_{+}(D)^{\f12}\hat{\D}_{q,k}\mathcal{H}_2\|_{L^2}^2\Big), 
 \end{split}
\end{equation*}
and hence
\begin{equation*}
 \begin{split}
&\|\Lambda\Lambda^{-2}\lambda_{+}(D)\hat{\D}_{q,k}w\|_{L^2}^2+\|\Lambda^{-1}\lambda_{+}(D)^{\f12}\hat{\D}_{q,k}\mathcal{H}_2\|_{L^2}^2
+2(\Lambda^{-2}\lambda_{+}(D)\hat{\D}_{q,k}\mathcal{H}_2|\Lambda^{-2}\lambda_{+}(D)\hat{\D}_{q,k}w)\\
&\quad\approx \|\Lambda\Lambda^{-2}\lambda_{+}(D)\hat{\D}_{q,k}w\|_{L^2}^2+\|\Lambda^{-1}\lambda_{+}(D)^{\f12}\hat{\D}_{q,k}\mathcal{H}_2\|_{L^2}^2\\
&\quad\ge C\Big(\|\Lambda^{-2}\lambda_{+}(D)\hat{\D}_{q,k}w\|_{L^2}^2+\|\hat{\D}_{q,k}\mathcal{H}_2\|_{L^2}^2\Big).  
 \end{split}
\end{equation*}
Here we lost the regularity of $\Lambda$ by using 
$$\|\Lambda\hat{\D}_{q,k}\phi\|_{L^2}\ge 2^q\|\hat{\D}_{q,k}\phi\|_{L^2}\ge 4\|\hat{\D}_{q,k}\phi\|_{L^2}.$$

On the other hand, one deduces from Lemma 5.1 in \cite{RD} that
\begin{equation}\label{a121}
\begin{split}
|\mathcal{Y}_{q,k}|&\le Cf_{q,k}\Big(\|\hat{\D}_{q,k}\mathcal{O}_2\|_{L^2}+\|\hat{\D}_{q,k}\mathcal{Q}_2\|_{L^2}
+\alpha_{q,k}\|\u\|_{\hat{B}^2}(\|\lambda_{+}(D)^{-\f12}\mathcal{H}_2\|_{\hat{B}^0\cap\hat{B}^1}+\|w\|_{\hat{B}^0})\Big).
\end{split}
\end{equation}

Combining \eqref{8111} and \eqref{a121}, one has
\begin{equation}\label{a131}
 \begin{split}
\f{d}{dt}f_{q,k}+f_{q,k}&\lesssim \|\hat{\D}_{q,k}\mathcal{O}_2\|_{L^2}+\|\hat{\D}_{q,k}\mathcal{Q}_2\|_{L^2}+\alpha_{q,k}\|\u\|_{\hat{B}^2}(\|\lambda_{+}(D)^{-\f12}\mathcal{H}_2\|_{\hat{B}^0\cap\hat{B}^1}+\|w\|_{\hat{B}^0})
 \end{split}
\end{equation} for some positive constant $\kappa$.

Summing \eqref{a131} over $q,k\in\mathbb{Z}$ and integrating over $t\in \R$, it follows
\begin{equation}\label{a151}
\begin{split}
&\sup_{0\le \tau<\infty}\sum_{q>1}\Big(\|\hat{\D}_{q,k}w(\tau)\|_{L^2}+\|\hat{\D}_{q,k}\mathcal{H}_2(\tau)\|_{L^2}\Big)+\sum_{q>1}\|\hat{\D}_{q,k}\mathcal{H}_2\|_{L^1(L^2)}\\
&\quad\le C\Big(\|w(0)\|_{\hat{B}^0}+\|\lambda_{+}(D)^{-\f12}\mathcal{H}_2(0)\|_{\hat{B}^0\cap\hat{B}^1}+\|\mathcal{O}_2\|_{L^1(\hat{B}^0)}+\|\lambda_{+}(D)^{-\f12}\mathcal{Q}_2\|_{L^1(\hat{B}^0\cap\hat{B}^1)}\\
&\qquad+\|\u\|_{L^1(\hat{B}^2)}(\|w\|_{L^\infty(\hat{B}^0)}+\|\lambda_{+}(D)^{-\f12}\mathcal{H}_2\|_{L^\infty(\hat{B}^0\cap\hat{B}^1)})\Big).
\end{split}
\end{equation}

\texttt{Step 2: Smoothing Effect.} We are going to use the parabolicity of $w$ to improve the dissipation of $w$. According to \eqref{891}, it is only left to handle the case $q>1$.
Going back to the equation \eqref{831b}, and regarding now the term $\mathcal{H}_2$ as an external term, one has
\begin{equation*}
\begin{split}
&\f12\f{d}{dt}\|\hat{\D}_{q,k}w\|_{L^2}^2+(\hat{\D}_{q,k}(\u\cdot\nabla w)|\hat{\D}_{q,k}w)+2^{2q}\|\hat{\D}_{q,k}w\|_{L^2}^2\\
&\quad\le C\Big(\|\hat{\D}_{q,k}\mathcal{H}_2\|_{L^2}+\|\hat{\D}_{q,k}\mathcal{O}_2\|_{L^2}\Big)\|\hat{\D}_{q,k}w\|_{L^2}, 
\end{split}
\end{equation*}
and hence
\begin{equation}\label{a161}
\begin{split}
\sum_{q>1}2^{2q}\|\hat{\D}_{q,k}w\|_{L^1(L^2)}&\le C\Big(\|w(0)\|_{\hat{B}^0}+\|\mathcal{O}_2\|_{L^1(\hat{B}^0)}+\|\u\|_{\hat{B}^2}\|w\|_{\hat{B}^0}\\
&\quad+\sum_{q>1}\|\hat{\D}_{q,k}\mathcal{H}_2\|_{L^1(L^2)}\Big).
\end{split}
\end{equation}
The desired estimate \eqref{821} follows from \eqref{891}, \eqref{a151} and \eqref{a161}.

\end{proof}

\bigskip\bigskip
%%%%%%%%%%%%%%%%%%%%%%%%%%%%%%%%%%%%%%%%%%%%%%%%%%%%%%%%%%%%%%%%%%%%%%%%%%%%%%%%%%%%%%%%%%%%%%%%%%%%%%%%%%%%%%%%%%%%%%%%%%%%%%%%%%%

\section{Dissipations of $\H$ and $b$}
This section is devoted to the dissipation estimates for the magnetic field $\H$ and the density $b$.
To this end, following line by line through \eqref{17}-\eqref{18}, using Plancheral's identity and \eqref{16d}, one has
\begin{equation*}
 \begin{split}
&\|\Lambda\hat{\D}_{q,k}\H\|_{L^2}^2+\left\|\Lambda \hat{\D}_{q,k}b\right\|_{L^2}^2\\
&\quad=\Big(\hat{\D}_{q,k}(\nabla(b+\H_1)-\partial_{x_1}\H)|\hat{\D}_{q,k}\Dv \mathcal{A}\Big)+\Big(\Lambda\hat{\D}_{q,k}b|\Lambda\hat{\D}_{q,k}\det\mathcal{A}\Big)\\
&\quad=\Big(\hat{\D}_{q,k}\mathcal{H}|\mathcal{P}\hat{\D}_{q,k}\Dv \mathcal{A}\Big)+\Big(\hat{\D}_{q,k}\mathcal{P}\nabla\det\mathcal{A}|\mathcal{P}\hat{\D}_{q,k}\Dv \mathcal{A}\Big)
+\Big(\Lambda\hat{\D}_{q,k}b|\Lambda\hat{\D}_{q,k}\det\mathcal{A}\Big),
 \end{split}
\end{equation*}
and hence 
\begin{equation}\label{323}
 \begin{split}
&\|\Lambda\hat{\D}_{q,k}\H\|_{L^2}^2+\left\|\Lambda\hat{\D}_{q,k}b\right\|_{L^2}^2\\
&\lesssim \Big(\hat{\D}_{q,k}\mathcal{H}|\mathcal{P}\hat{\D}_{q,k}\Dv \mathcal{A}\Big)+\Big(\hat{\D}_{q,k}\mathcal{P}\nabla\det\mathcal{A}|\mathcal{P}\hat{\D}_{q,k}\Dv \mathcal{A}\Big)
+\|\Lambda\hat{\D}_{q,k}\det\mathcal{A}\|^2_{L^2}.
 \end{split}
\end{equation}
We claim first that
\begin{Lemma}\label{l21}
There holds true
\begin{equation}\label{320a}
 \begin{split}
&\left(\int_0^\infty\left(\sum_{q,k\in\mathbb{Z}}\left|\Big(\hat{\D}_{q,k}\mathcal{H}|\mathcal{P}\hat{\D}_{q,k}\Dv \mathcal{A}\Big)\right|^{\f12}\right)^2dt\right)^{\f12}\\
&\lesssim \|\mathcal{A}\|_{L^\infty(\hat{B}^{1})}+\|\det\mathcal{A}\|_{L^2(\hat{B}^1)}+\|\lambda_{-}(D)^{-\f12}\mathcal{H}_1(0)\|_{\tilde{B}^{0,1}}
+\|\lambda_{+}(D)^{-\f12}\mathcal{H}_2(0)\|_{\hat{B}^{0}\cap\hat{B}^1}\\
&\qquad+\|\u(0)\|_{\hat{B}^0}+\|\lambda_{-}(D)^{-\f12}\mathcal{Q}_1\|_{L^1(\tilde{B}^{0,1})}+\|\lambda_{+}(D)^{-\f12}\mathcal{Q}_2\|_{L^1(\hat{B}^{0}\cap\hat{B}^1)}+\|\mathcal{O}\|_{L^1(\hat{B}^{0})}\\
&\qquad+\|\u\|_{L^1(\hat{B}^2)}(\|\lambda_{-}(D)^{-\f12}\mathcal{H}_1\|_{L^\infty(\tilde{B}^{0,1})}+\|\lambda_{+}(D)^{-\f12}\mathcal{H}_2\|_{L^\infty(\hat{B}^{0}\cap\hat{B}^1)}+\|\u\|_{L^\infty(\hat{B}^0)}).
 \end{split}
\end{equation}
\end{Lemma}
\begin{proof}
Note that
\begin{subequations}\label{324}
 \begin{align}
&\int_0^\infty\left(\sum_{q,k\in\mathbb{Z}}\left|\Big(\hat{\D}_{q,k}\mathcal{H}|\mathcal{P}\hat{\D}_{q,k}\Dv \mathcal{A}\Big)\right|^{\f12}\right)^2dt\nonumber\\
&\quad\lesssim \int_0^\infty\left(\sum_{q,k\in\mathbb{Z}}\left|\Big(\hat{\D}_{q,k}\mathcal{H}_1|\Big(\mathcal{P}\hat{\D}_{q,k}\Dv \mathcal{A}\Big)_1\Big)\right|^{\f12}\right)^2dt\label{324a}\\
&\qquad+\int_0^\infty\left(\sum_{q,k\in\mathbb{Z}}\left|\Big(\hat{\D}_{q,k}\mathcal{H}_2|\Big(\mathcal{P}\hat{\D}_{q,k}\Dv \mathcal{A}\Big)_2\Big)\right|^{\f12}\right)^2dt.\label{324b}
 \end{align}
\end{subequations}
Since $\Dv\mathcal{A}=\D\alpha$, one deduces from \eqref{le2} that
\begin{equation}\label{324c1}
 \begin{split}
&(\hat{\D}_{q,k}\mathcal{H}_1|(\mathcal{P}\hat{\D}_{q,k}\Dv\mathcal{A})_1)=(\hat{\D}_{q,k}\mathcal{H}_1|\D\hat{\D}_{q,k}(\mathcal{P}\alpha)_1)\\
&\quad=\|\Lambda\lambda_{-}(D)^{-\f12}\hat{\D}_{q,k}\mathcal{H}_1\|_{L^2}^2-(\Lambda\lambda_{-}^{-\f12}(D)\hat{\D}_{q,k}(\mathcal{P}\nabla\det\mathcal{A})_1| \Lambda\lambda_{-}^{-\f12}(D)\hat{\D}_{q,k}\mathcal{H}_1)
 \end{split}
\end{equation}
and
\begin{equation}\label{324c2}
 \begin{split}
&(\hat{\D}_{q,k}\mathcal{H}_2|(\mathcal{P}\hat{\D}_{q,k}\Dv\mathcal{A})_2)=(\hat{\D}_{q,k}\mathcal{H}_2|\D\hat{\D}_{q,k}(\mathcal{P}\alpha)_2)\\
&\quad=\|\Lambda\lambda_{+}(D)^{-\f12}\hat{\D}_{q,k}\mathcal{H}_2\|_{L^2}^2-(\Lambda\lambda_{+}^{-\f12}(D)\hat{\D}_{q,k}(\mathcal{P}\nabla\det\mathcal{A})_2| \Lambda\lambda_{+}^{-\f12}(D)\hat{\D}_{q,k}\mathcal{H}_2).
 \end{split}
\end{equation}

\texttt{Estimate of $\eqref{324a}$.} We consider two cases: $k+1\ge 2q$ and $k+1<2q$. 

\textit{Case 1: $k+1\ge 2q$.} 
Thanks to \eqref{le22}, one has
$$\|\Lambda\lambda_{-}^{-\f12}(D)\hat{\D}_{q,k}(\mathcal{P}\nabla\det\mathcal{A})_1\|_{L^2}\lesssim \|\Lambda\hat{\D}_{q,k}\det\mathcal{A}\|_{L^2}.$$
Thus, due to \eqref{324c1}, the Cauchy-Schwarz inequality and Holder's inequality for series imply
\begin{equation*}
 \begin{split}
\eqref{324a}&\le \int_0^\infty\left(\sum_{k+1\ge 2q}2^q\|\lambda_{-}(D)^{-\f12}\hat{\D}_{q,k}\mathcal{H}_1\|_{L^2}\right)^2dt+\|\det\mathcal{A}\|_{L^2(\hat{B}^1)}^2\\
&=\int_0^\infty\left(\sum_{k+1\ge 2q}2^q\|\hat{\D}_{q,k}\lambda_{-}(D)^{-\f12}\mathcal{H}_1\|_{L^2}^{1/2}\|\hat{\D}_{q,k}\lambda_{-}(D)^{-\f12}\mathcal{H}_1\|_{L^2}^{1/2}\right)^2dt\\
&\quad+\|\det\mathcal{A}\|_{L^2(\hat{B}^1)}^2\\
&\lesssim \int_0^\infty\left(\sum_{k+1\ge 2q}2^{2q}\|\hat{\D}_{q,k}\lambda_{-}(D)^{-\f12}\mathcal{H}_1\|_{L^2}\right)\left(\sum_{k+1\ge 2q}\|\hat{\D}_{q,k}\lambda_{-}(D)^{-\f12}\mathcal{H}_1\|_{L^2}\right)dt\\
&\quad+\|\det\mathcal{A}\|_{L^2(\hat{B}^1)}^2\\
&\lesssim \|\lambda_{-}(D)^{-\f12}\mathcal{H}_1\|_{L^\infty(\tilde{B}^{0,1})}\int_0^\infty\left(\sum_{k+1\ge 2q}2^{2q}\|\hat{\D}_{q,k}\lambda_{-}(D)^{-\f12}\mathcal{H}_1\|_{L^2}\right)dt\\
&\quad+\|\det\mathcal{A}\|_{L^2(\hat{B}^1)}^2,
 \end{split}
\end{equation*}
and the desired estimate \eqref{320a} follows from Proposition \ref{p00}.

\textit{Case 2: $k+1<2q$.} In this case, the Cauchy-Schwarz inequality yields
\begin{equation*}
 \begin{split}
&\eqref{324a}\lesssim\int_0^\infty\left(\sum_{k+1< 2q}\|\hat{\D}_{q,k}\mathcal{H}_1\|^{\f12}_{L^2}\|\hat{\D}_{q,k}\Dv\mathcal{A}\|_{L^2}^{\f12}\right)^2dt\\
&\lesssim \int_0^\infty\left(\sum_{k+1< 2q}\|\hat{\D}_{q,k}\mathcal{H}_1\|_{L^2}\right)\left(\sum_{k+1< 2q}\|\hat{\D}_{q,k}\Dv\mathcal{A}\|_{L^2}\right)dt\\
&\lesssim \|\mathcal{A}\|_{L^\infty(\hat{B}^{1})}\int_0^\infty\left(\sum_{k+1< 2q}\|\hat{\D}_{q,k}\mathcal{H}_1\|_{L^2}\right)dt,
 \end{split}
\end{equation*}
and the desired estimate \eqref{320a} follows from Proposition \ref{p00}.

\texttt{Estimate of $\eqref{324b}$.} We consider two cases: $q\le 1$ and $q>1$.

\textit{Case 1: $q\le 1$.} 
According to \eqref{324c2}, Holder's inequality for series implies
\begin{equation*}
 \begin{split}
\eqref{324b}&\le\int_0^\infty\left(\sum_{q\le 1}\|\Lambda\lambda_{+}(D)^{-\f12}\hat{\D}_{q,k}\mathcal{H}_2\|_{L^2}\right)^2dt+\|\det\mathcal{A}\|_{L^2(\hat{B}^1)}^2\\
&\lesssim\int_0^\infty\left(\sum_{q\le 1}2^q\|\hat{\D}_{q,k}\lambda_{+}(D)^{-\f12}\mathcal{H}_2\|_{L^2}^{1/2}\|\hat{\D}_{q,k}\lambda_{+}(D)^{-\f12}\mathcal{H}_2\|_{L^2}^{1/2}\right)^2dt\\
&\quad+\|\det\mathcal{A}\|_{L^2(\hat{B}^1)}^2\\
&\lesssim \int_0^\infty\left(\sum_{q\le 1}2^{2q}\|\hat{\D}_{q,k}\lambda_{+}(D)^{-\f12}\mathcal{H}_2\|_{L^2}\right)\left(\sum_{q\le 1}\|\hat{\D}_{q,k}\lambda_{+}(D)^{-\f12}\mathcal{H}_2\|_{L^2}\right)dt\\
&\quad+\|\det\mathcal{A}\|_{L^2(\hat{B}^1)}^2\\
&\lesssim \|\lambda_{+}(D)^{-\f12}\mathcal{H}_2\|_{L^\infty(\hat{B}^{0}\cap\hat{B}^1)}\int_0^\infty\left(\sum_{q\le 1}2^{2q}\|\hat{\D}_{q,k}\lambda_{+}(D)^{-\f12}\mathcal{H}_2\|_{L^2}\right)dt\\
&\quad+\|\det\mathcal{A}\|_{L^2(\hat{B}^1)}^2,
 \end{split}
\end{equation*}
and the desired estimate \eqref{320a} follows from Proposition \ref{p111}.

\textit{Case 2: $q>1$.} In this case, the Cauchy-Swartz inequality yields
\begin{equation*}
 \begin{split}
&\eqref{324b}\lesssim\int_0^\infty\left(\sum_{q>1}\|\hat{\D}_{q,k}\mathcal{H}_2\|^{\f12}_{L^2}\|\hat{\D}_{q,k}\Dv\mathcal{A}\|_{L^2}^{\f12}\right)^2dt\\
&\lesssim \int_0^\infty\left(\sum_{q>1}\|\hat{\D}_{q,k}\mathcal{H}_2\|_{L^2}\right)\left(\sum_{q>1}\|\hat{\D}_{q,k}\Dv\mathcal{A}\|_{L^2}\right)dt\\
&\lesssim \|\mathcal{A}\|_{L^\infty(\hat{B}^{1})}\int_0^\infty\left(\sum_{q>1}\|\hat{\D}_{q,k}\mathcal{H}_2\|_{L^2}\right)dt,
 \end{split}
\end{equation*}
and the desired estimate \eqref{320a} follows from Proposition \ref{p111}.
\end{proof}

Now it is ready to state the dissipations for $b$ and $\H$.
\begin{Lemma}\label{l2}
Assume that $\sup_{t\in[0,T]}\|\mathcal{A}(t)\|_{\hat{B}^{1}}\le \epsilon$ for $0\le T\le\infty$ and sufficiently small $\epsilon$. Then for solutions $(b,\u,\H)$ of \eqref{e2}, there holds
\begin{equation}\label{327}
 \begin{split}
&\left\|b\right\|_{L_T^2(\hat{B}^1)}+\left\|\H\right\|_{L_T^2(\hat{B}^1)}\\
&\quad\lesssim \|\mathcal{A}\|_{L^\infty_T(\hat{B}^1)}+\|\u(0)\|_{\hat{B}^0}+\|\lambda_{-}(D)^{-\f12}\mathcal{H}_1(0)\|_{\tilde{B}^{0,1}}+\|\lambda_{+}(D)^{-\f12}\mathcal{H}_2(0)\|_{\hat{B}^{0}\cap\hat{B}^1}\\
&\qquad+\|\lambda_{-}(D)^{-\f12}\mathcal{Q}_1\|_{L^1(\tilde{B}^{0,1})}+\|\lambda_{+}(D)^{-\f12}\mathcal{Q}_2\|_{L^1(\hat{B}^{0}\cap\hat{B}^1)}+\|\mathcal{O}\|_{L^1(\hat{B}^{0})}\\
&\qquad+\|\u\|_{L^1(\hat{B}^2)}\Big(\|\lambda_{-}(D)^{-\f12}\mathcal{H}_1\|_{L^\infty(\tilde{B}^{0,1})}+\|\lambda_{+}(D)^{-\f12}\mathcal{H}_2\|_{L^\infty(\hat{B}^{0}\cap\hat{B}^1)}+\|\u\|_{L^\infty(\hat{B}^0)}\Big).
 \end{split}
\end{equation}
\end{Lemma}
\begin{proof}
Since $\mathcal{P}^2=I$ and $\mathcal{P}^\top=\mathcal{P}$, one has, using \eqref{16b}
\begin{equation*}
 \begin{split}
\Big(\hat{\D}_{q,k}\mathcal{P}\nabla\det\mathcal{A}|\mathcal{P}\hat{\D}_{q,k}\Dv \mathcal{A}\Big)&=\Big(\hat{\D}_{q,k}\nabla\det\mathcal{A}|\hat{\D}_{q,k}\Dv \mathcal{A}\Big)\\
&=-\Big(\hat{\D}_{q,k}\det\mathcal{A}|\D\hat{\D}_{q,k}\tr \mathcal{A}\Big)\\
&=\Big(\Lambda\hat{\D}_{q,k}\det\mathcal{A}|\Lambda\hat{\D}_{q,k}b\Big)-\Big\|\Lambda\hat{\D}_{q,k}\det\mathcal{A}\Big\|^2_{L^2}.
 \end{split}
\end{equation*}
This identity, combining with \eqref{323}, yields
\begin{equation}\label{326a}
 \begin{split}
&\left\|b\right\|_{L_T^2(\hat{B}^1)}^2+\left\|\H\right\|_{L_T^2(\hat{B}^1)}^2\\
&\quad\lesssim \int_0^\infty\left(\sum_{q,k\in\mathbb{Z}}\left|\Big(\hat{\D}_{q,k}\mathcal{P}\mathcal{H}|\mathcal{P}\hat{\D}_{q,k}\Dv \mathcal{A}\Big)\right|^{\f12}\right)^2dt
+\|\det\mathcal{A}\|_{L^2_T(\hat{B}^1)}^2.
 \end{split}
\end{equation}

On the other hand, the identity \eqref{16a} implies
\begin{equation*}
 \begin{split}
\det\mathcal{A}=\mathcal{A}_{11}\mathcal{A}_{22}-\mathcal{A}_{12}\mathcal{A}_{21}=\mathcal{A}_{11}\H_1+\mathcal{A}_{12}\H_2,
 \end{split}
\end{equation*}
and hence Proposition \ref{p2} yields
\begin{equation}\label{326}
\begin{split}
 \|\det\mathcal{A}\|_{L^2_T(\hat{B}^1)}\lesssim\|\mathcal{A}\|_{L_T^\infty(\hat{B}^1)}\|\H\|_{L^2_T(\hat{B}^1)}.
\end{split}
\end{equation}
The desired estimate \eqref{327} then follows from \eqref{326a}, \eqref{326}, Lemma \ref{l21}, and the assumption \\$\sup_{t\in[0,T]}\|\mathcal{A}(t)\|_{\hat{B}^{1}}\le \epsilon$ with $C\epsilon^2\le 1/2$.
\end{proof}

\bigskip\bigskip
%%%%%%%%%%%%%%%%%%%%%%%%%%%%%%%%%%%%%%%%%%%%%%%%%%%%%%%%%%%%%%%%%%%%%%%%%%%%%%%%%%%%%%%%%%%%%%%%%%%%%%%%%%%%%%%%%%%%%%%%%%%%%%%%

\section{Proof of Theorem \ref{MT}}

This section aims at the proof of Theorem \ref{MT}, more precisely the proof of ($\mathfrak{G}$). To begin with, we establish the estimate for $\|\mathcal{A}\|_{L^\infty_t(\hat{B}^{1})}$.
\begin{Lemma}\label{l6a}
\begin{equation*}
\|\mathcal{A}\|_{L^\infty_t(\hat{B}^{1})}\lesssim X(0)+X(t)^2.
\end{equation*}
\end{Lemma}
\begin{proof}
From \eqref{17d}, the function $\mathcal{A}$ satisfies a tranport equation
\begin{equation*}
 \partial_t\mathcal{A}+\u\cdot\nabla\mathcal{A}+\nabla\u=-\mathcal{A}\nabla\u,
\end{equation*}
and hence according to Lemma \ref{cl}, there holds 
\begin{equation}\label{65}
 \begin{split}
\f12\f{d}{dt}\|\Lambda\hat{\D}_{q,k}\mathcal{A}\|_{L^2}^2&=-(\Lambda\hat{\D}_{q,k}\nabla\u|\Lambda\hat{\D}_{q,k}\mathcal{A})-(\Lambda\hat{\D}_{q,k}(\u\cdot\nabla\mathcal{A})|\Lambda\hat{\D}_{q,k}\mathcal{A})\\
&\quad-(\Lambda\hat{\D}_{q,k}(\mathcal{A}\nabla\u)|\Lambda\hat{\D}_{q,k}\mathcal{A})\\
&\lesssim \|\Lambda\hat{\D}_{q,k}\mathcal{A}\|_{L^2}\Big(\|\Lambda\hat{\D}_{q,k}\nabla\u\|_{L^2}+\|\Lambda\hat{\D}_{q,k}(\mathcal{A}\nabla\u)\|_{L^2}\\
&\quad+\alpha_{q,k}\|\u\|_{\hat{B}^2}\|\mathcal{A}\|_{\hat{B}^{1}}\Big).
\end{split}
\end{equation}
Summing \eqref{65} over $q,q_1,q_2\in\mathbb{Z}$ and integrating over $t$, one obtains
\begin{equation}\label{65b}
 \begin{split}
\sup_{t\ge 0}\|\mathcal{A}(t)\|_{\hat{B}^1}&\lesssim \|\mathcal{A}_0\|_{\hat{B}^1}
+\|\nabla\u\|_{L^1(\hat{B}^{1})}+\|\mathcal{A}\nabla\u\|_{L^1(\hat{B}^{1})}+\|\u\|_{L^1(\hat{B}^2)}\|\mathcal{A}\|_{L^\infty(\hat{B}^{1})}.
 \end{split}
\end{equation}
One has
\begin{equation*}
 \begin{split}
\|\mathcal{A}\nabla\u\|_{L^1(\hat{B}^{1})}&\lesssim\|\mathcal{A}\|_{L^\infty(\hat{B}^{1})}\|\nabla\u\|_{L^1(\hat{B}^1)}\lesssim\|\mathcal{A}\|_{L^\infty(\hat{B}^{1})}\|\u\|_{L^1(\hat{B}^2)}\lesssim X(t)^2.
 \end{split}
\end{equation*}
Substituting this into \eqref{65b}, one has
\begin{equation*}
 \begin{split}
\|\mathcal{A}(t)\|_{\hat{B}^{1}}&\lesssim \|\mathcal{A}_0\|_{\hat{B}^{1}}+X(t)^2.
 \end{split}
\end{equation*}
Taking the $L^\infty$ norm over time gives the desired estimate.
\end{proof}

From now on, due to the smallness of the initial data and \eqref{le8}, one can assume that
$$\|\r\|_{L^\infty(L^\infty)}\lesssim 1.$$

Next we turn to the estimate for the remaining part of $X(t)$.
\begin{Lemma}\label{l6}
\begin{equation*}
 \begin{split}
 X(t)-\|\mathcal{A}\|_{L^\infty_t(\hat{B}^{1})}\lesssim X(0)+X(t)^2. 
 \end{split}
\end{equation*}
\end{Lemma}
\begin{proof}
In view of Proposition \ref{p00}, Proposition \ref{p111}, Lemma \ref{l2} and Lemma \ref{l6a}, one has
\begin{equation}\label{61}
\begin{split}
& X(t)-\|\mathcal{A}\|_{L^\infty_t(\hat{B}^{1})}\\
&\quad\le C\Big(X(0)+\|\mathcal{O}\|_{L^1(\hat{B}^{0})}+\|\lambda_{-}(D)^{-\f12}\mathcal{Q}_1\|_{L^1(\tilde{B}^{0,1})}+\|\lambda_{+}(D)^{-\f12}\mathcal{Q}_2\|_{L^1(\hat{B}^{0}\cap\hat{B}^1)}+X(t)^2\Big).
\end{split}
\end{equation}

\texttt{Estimate of $\|\mathcal{O}\|_{L^1(\hat{B}^0)}$.} The estimate for $-\mathcal{P}(\u\cdot\nabla\u)+\u\cdot\nabla\mathcal{P}\u$ is straightforward since
\begin{equation}\label{62a}
 \begin{split}
\|-\mathcal{P}(\u\cdot\nabla\u)+\u\cdot\nabla\mathcal{P}\u\|_{L^1(\hat{B}^0)}&\lesssim \|\u\|_{L^2(\hat{B}^1)}\|\nabla\u\|_{L^2(\hat{B}^0)}\\
&\lesssim \|\u\|_{L^2(\hat{B}^1)}^2\lesssim \|\u\|_{L^1(\hat{B}^2)}\|\u\|_{L^\infty(\hat{B}^0)}\lesssim X(t)^2
 \end{split}
\end{equation}
since the operator $\mathcal{P}(\xi)$ is orthonormal.

We are left with the estimate of $\|\mathcal{P}\mathcal{L}\|_{L^1(\hat{B}^0)}$, which is equivalent to the estimate of $\|\mathcal{L}\|_{L^1(\hat{B}^0)}$. The estimates for $\H\cdot\nabla\H$ and $\nabla|\H|^2$ 
are straightforward since
\begin{equation*}
\|\H\cdot\nabla\H\|_{L^1(\hat{B}^0)}+\|\nabla|\H|^2\|_{L^1(\hat{B}^0)}\lesssim\|\H\|_{L^2(\hat{B}^{1})}\|\nabla\H\|_{L^2(\hat{B}^{0})}\lesssim \|\H\|_{L^2(\hat{B}^{1})}^2\lesssim X(t)^2.
\end{equation*}
Therefore
\begin{equation*}
\left\|\f{1}{\r}\left(\H\cdot\nabla\H-\f12\nabla|\H|^2\right)\right\|_{L^1(\hat{B}^0)}\lesssim X(t)^2.
\end{equation*}

Similarly, one has
$$\|b\nabla b\|_{L^1(\hat{B}^0)}\lesssim\|b\|_{L^2(\hat{B}^{1})}\|\nabla b\|_{L^2(\hat{B}^{0})}\lesssim\|b\|_{L^2(\hat{B}^{1})}^2\lesssim X(t)^2,$$
and
\begin{equation*}
 \begin{split}
\left\|\left(\f{1}{\rho}-1\right)(\nabla\H_1-\partial_{x_1}\H)\right\|_{L^1(\hat{B}^0)}&\lesssim\|b\|_{L^2(\hat{B}^{1})}\|\nabla \H\|_{L^2(\hat{B}^{0})}\\
&\lesssim\|b\|_{L^2(\hat{B}^{1})}\|\H\|_{L^2(\hat{B}^{1})}\lesssim X(t)^2.
 \end{split}
\end{equation*}

On the other hand, using \eqref{le8}, one has
\begin{equation*}
\begin{split}
 \left\|\left(\f{1}{\rho}-1\right)(\mu\D\u+(\lambda+\mu)\nabla\Dv\u)\right\|_{L^1(\hat{B}^0)}&\lesssim \|b\|_{L^\infty(\hat{B}^{1})}\|\D\u\|_{L^1(\hat{B}^0)}\\
&\lesssim \|b\|_{L^\infty(\hat{B}^0\cap\hat{B}^{1})}\|\u\|_{L^1(\hat{B}^2)}\lesssim X(t)^2.
\end{split}
\end{equation*}

Adding all those estimates above together yields
$$\|\mathcal{L}\|_{L^1(B^0)}\lesssim X(t)^2,$$
which, combing with \eqref{62a}, gives
\begin{equation}\label{62}
\|\mathcal{O}\|_{L^1(\hat{B}^0)}\lesssim X(t)^2. 
\end{equation}

Before we move to the estimate relating to $\mathcal{Q}$, it is better to take a look at the structure of $\mathcal{Q}$ carefully due to the degenerated hyperbolicity. In other words the structure of $(\mathcal{P}\mathcal{Q})_1$
plays an important role in the estimation. Indeed, one easily sees
\begin{subequations}
 \begin{align}
&\mathcal{F}\Big[\mathcal{P}\Big(\nabla(\mathcal{K}+\mathcal{M}_1)-\partial_{x_1}\mathcal{M}\Big)\Big]_1\nonumber\\
&\quad=\f{1}{\sqrt{\xi_2^2+\xi_1^2\left(\f12+r(\xi)\right)^2}}\left(-\xi_2,\f12\xi_1\right)\cdot\mathcal{F}\Big(\nabla(\mathcal{K}+\mathcal{M}_1)-\partial_{x_1}\mathcal{M}\Big)\nonumber\\
&\qquad+\f{1}{\sqrt{\xi_2^2+\xi_1^2\left(\f12+r(\xi)\right)^2}}\xi_1r(\xi)\mathcal{F}\Big(\partial_{x_2}(\mathcal{K}+\mathcal{M}_1)-\partial_{x_1}\mathcal{M}_2\Big)\nonumber\\
&\quad=\f{1}{2\sqrt{\xi_2^2+\xi_1^2\left(\f12+r(\xi)\right)^2}}\Big(\xi_2\xi_1\mathcal{F}(\mathcal{M}_1-\mathcal{K})-\xi_1^2\mathcal{F}\mathcal{M}_2\Big)\label{6100a}\\
&\qquad+\f{1}{\sqrt{\xi_2^2+\xi_1^2\left(\f12+r(\xi)\right)^2}}\xi_1r(\xi)\Big(\xi_2\mathcal{F}(\mathcal{K}+\mathcal{M}_1)-\xi_1\mathcal{F}\mathcal{M}_2\Big)\label{6100b}.
 \end{align}
\end{subequations}
Note that since $r(\xi)$ contains a factor $\xi_1^2/|\xi|^2$, the expression \eqref{6100b} is actually less than $|\xi|^{-1}\xi_1^2|\mathcal{F}(\mathcal{K}+\mathcal{M})|$ and hence is easier to handle than \eqref{6100a}.

\texttt{Estimate of $\|\lambda_{-}(D)^{-\f12}\mathcal{Q}_1\|_{L^1(\tilde{B}^{0,1})}$.} We first control the term relating to $\nabla(\mathcal{K}+\mathcal{M}_1)-\partial_{x_1}\mathcal{M}$. Indeed, for the term associated with
\eqref{6100b}, one has, using \eqref{27a} and \eqref{le8}
\begin{equation}\label{6101}
 \begin{split}
\|\textrm{Terms associated with}\quad \eqref{6100b}\|_{L^1(\tilde{B}^{0,1})}&\lesssim \|\lambda(D)^{\f12}\Lambda^{-1}(\mathcal{K}+\mathcal{M})\|_{L^1(\tilde{B}^{0,1})}\\
&\lesssim \|\mathcal{K}+\mathcal{M}\|_{L^1(\hat{B}^{0}\cap\hat{B}^1)}\\
&\lesssim \|\nabla\u\|_{L^1(\hat{B}^1)}\Big(\|b\|_{L^1(\hat{B}^0\cap\hat{B}^1)}+\|\H\|_{L^1(\hat{B}^0\cap\hat{B}^1)}\Big)\\
&\lesssim X(t)^2.
 \end{split}
\end{equation}
Note that
$$\mathcal{M}_1-\mathcal{K}=(b-\H_1)\Dv\u+\H\cdot\nabla\u_1,$$ terms associated with \eqref{6100a} actually contains three cases
\begin{subequations}
 \begin{align}
&\textrm{Terms associated with}\quad \eqref{6100a}\nonumber\\
&\quad=-\f{1}{2\sqrt{\xi_2^2+\xi_1^2\left(\f12+r(\xi)\right)^2}}\xi_1^2\mathcal{F}\mathcal{M}_2\label{6102a}\\
&\qquad+\f{1}{2\sqrt{\xi_2^2+\xi_1^2\left(\f12+r(\xi)\right)^2}}\xi_2\xi_1\mathcal{F}\Big((b-\H_1)\Dv\u+\H_2\partial_{x_2}\u_1\Big)\label{6102b}\\
&\qquad+\f{1}{2\sqrt{\xi_2^2+\xi_1^2\left(\f12+r(\xi)\right)^2}}\xi_2\xi_1\mathcal{F}\Big(\H_1\partial_{x_1}\u_1\Big)\label{6102c}.
 \end{align}
\end{subequations}
One can treat terms associated with \eqref{6102a} similarly as \eqref{6101} to get
\begin{equation}\label{6103}
\|\textrm{Terms associated with}\quad \eqref{6102a}\|_{L^1(\tilde{B}^{0,1})}\lesssim X(t)^2;
\end{equation}
and for terms associated with \eqref{6102b} one has, using Lemma \ref{lee} and Lemma \ref{le}
\begin{equation}\label{6104}
 \begin{split}
&\|\textrm{Terms associated with}\quad \eqref{6102b}\|_{L^1(\tilde{B}^{0,1})}\\
&\quad\lesssim \|(b-\H_1)\Dv\u+\H_2\partial_{x_2}\u_1\|_{L^1(\tilde{B}^{0,1})}\\
&\quad\lesssim \|\nabla\u\|_{L^1(\hat{B}^1)}\Big(\|b-\H_1\|_{L^\infty(\tilde{B}^{0,1})}+\|\H_2\|_{L^\infty(\tilde{B}^{0,1})}\Big)\\
&\quad\lesssim X(t)^2.
 \end{split}
\end{equation}
For terms associated with \eqref{6102c}, one has, using \eqref{27a} and \eqref{27b}
\begin{equation}\label{6105}
\begin{split}
&\|\textrm{Terms associated with}\quad \eqref{6102c}\|_{L^1(\tilde{B}^{0,1})}\\
&\quad\lesssim \|\H_1\partial_{x_1}\u_1\|_{L^1(\tilde{B}^{0,1})}\\
&\quad\approx \|\H_1\partial_{x_1}\u_1\|_{L^1(\hat{B}^{0})}+\|\H_1\partial_{x_1}\u_1\|_{L^1(\check{B}^{1})}\\
&\quad\lesssim \|\H\|_{L^\infty(\hat{B}^0)}\|\nabla\u\|_{L^1(\hat{B}^1)}+\|\H\|_{L^\infty(\hat{B}^1)}\|\partial_{x_1}\u_1\|_{L^1(\check{B}^1)}\\
&\quad\lesssim \|\H\|_{L^\infty(\hat{B}^0\cap\hat{B}^1)}\|\u\|_{L^1(\hat{B}^2)}\lesssim \|\H\|_{L^\infty(\tilde{B}^{0,1})}\|\u\|_{L^1(\hat{B}^2)}\lesssim X(t)^2.
\end{split}
\end{equation}
Summaring \eqref{6101}-\eqref{6105}, one has
\begin{equation}\label{6106}
\left\|\lambda_{-}(D)^{-\f12}\Big[\mathcal{P}\Big[\nabla(\mathcal{K}+\mathcal{M}_1)-\partial_{x_1}\mathcal{M}\Big]\Big]_1\right\|_{L^1(\tilde{B}^{0,1})}\lesssim X(t)^2.
\end{equation}
On the other hand, applying Lemma \ref{cl1}, one obtains
\begin{equation*}
\begin{split}
&\left\|\lambda_{-}(D)^{-\f12}\left[\u\cdot\nabla\mathcal{H}-\mathcal{P}\Big[\nabla\Big(\u\cdot\nabla(b+\H_1)\Big)-\partial_{x_1}(\u\cdot\nabla\H)\Big]\right]_1\right\|_{L^1(\tilde{B}^{0,1})}\\
&\quad\lesssim \|\nabla\u\|_{L^1(\hat{B}^1)}\|\lambda_{-}(D)^{-\f12}\mathcal{H}_1\|_{L^1(\tilde{B}^{0,1})}\lesssim X(t)^2,
\end{split}
\end{equation*}
which, combining with \eqref{6106}, gives
\begin{equation}\label{6107}
 \|\lambda_{-}(D)^{-\f12}\mathcal{Q}_1\|_{L^1(\tilde{B}^{0,1})}\lesssim X(t)^2.
\end{equation}

\texttt{Estimate of $\|\lambda_{+}(D)^{-\f12}\mathcal{Q}_2\|_{L^1(\hat{B}^{0}\cap\hat{B}^1)}$.} This estimate is much easier than the estimate for $\lambda_{-}(D)^{-\f12}\mathcal{Q}_1$ due to the fact 
$\lambda_{+}(\xi)\approx |\xi|^2$. Indeed, since $\lambda_{+}(\xi)^{-\f12}\approx |\xi|^{-1}$ and the operaor $\mathcal{P}$ is bounded in $L^2$, one has
\begin{equation}\label{6108}
\begin{split}
& \left\|\lambda_{+}(D)^{-\f12}\Big[\mathcal{P}\Big[\nabla(\mathcal{K}+\mathcal{M}_1)-\partial_{x_1}\mathcal{M}\Big]\Big]_2\right\|_{L^1(\hat{B}^{0}\cap\hat{B}^1)}\\
&\quad\lesssim \|\mathcal{K}+\mathcal{M}\|_{L^1(\hat{B}^{0}\cap\hat{B}^1)}\\
&\quad\lesssim \|\nabla\u\|_{L^1(\hat{B}^1)}\Big(\|b\|_{L^1(\hat{B}^0\cap\hat{B}^1)}+\|\H\|_{L^1(\hat{B}^0\cap\hat{B}^1)}\Big)\\
&\quad\lesssim X(t)^2.
\end{split}
\end{equation}
Applying Lemma \ref{cl1} again, one has
\begin{equation*}
 \begin{split}
&\left\|\lambda_{+}(D)^{-\f12}\left[\u\cdot\nabla\mathcal{H}-\mathcal{P}\Big[\nabla\Big(\u\cdot\nabla(b+\H_1)\Big)-\partial_{x_1}(\u\cdot\nabla\H)\Big]\right]_2\right\|_{L^1(\hat{B}^{0}\cap\hat{B}^1)}\\
&\quad\lesssim \|\nabla\u\|_{L^1(\hat{B}^1)}\|\lambda_{+}(D)^{-\f12}\mathcal{H}_2\|_{L^1(\hat{B}^{0}\cap\hat{B}^1)}\lesssim X(t)^2,  
 \end{split}
\end{equation*}
which, combining together with \eqref{6108}, gives
\begin{equation}\label{64}
 \begin{split}
\|\lambda_{+}(D)^{-\f12}\mathcal{Q}_2\|_{L^1(\hat{B}^{0}\cap\hat{B}^1)}\lesssim X(t)^2.
 \end{split}
\end{equation}

Summaring \eqref{61}, \eqref{62}, \eqref{6107} and \eqref{64} together yields the desired estimate.
\end{proof}

\bigskip\bigskip
%%%%%%%%%%%%%%%%%%%%%%%%%%%%%%%%%%%%%%%%%%%%%%%%%%%%%%%%%%%%%%%%%%%%%%%%%%%%%%%%%%%%%%%%%%%%%%%%%%%%%%%%%%%%%%%%%%%%%%%%%%%%

\section{Appendix: Proofs of Product Laws}

This appendix is devoted to the proof of Proposition \ref{p2} and Lemma \ref{cl}. The ideas to prove them are not new (see for example \cite{RD}), but the proof requires paradifferential calculus. 
The isotropic para-differential decomposition of Bony form in $\R^2$ can be stated as follows: let $f,g\in\mathcal{S}'(\R^2)$,
$$fg=T(f,g)+\bar{T}(f,g)+R(f,g)$$ with $\bar{T}(f,g)=T(g,f)$ and
$$T(f,g)\overset{def}=\sum_{j\in\mathbb{Z}}S_{j-1}f\D_jg,\quad R(f,g)\overset{def}=\sum_{j\in\mathbb{Z}}\D_jf\tilde{\D}_jg,\quad \tilde{\D}_jg\overset{def}=\sum_{l=j-1}^{j+1}\D_lg.$$
We use $T^i $, $\bar{T}^i$ and $R^i$ to denote the isotropic para-differential decomposition of Bony form for $\R$ in the direction of $x_i$ for $i=1,2$ respectively.

We first give the proof of Proposition \ref{p2}.
\begin{proof}[Proof of Proposition \ref{p2}]
By Bony's decomposition, one has
\begin{equation}\label{71}
 fg=\Big(TT^1+T\bar{T}^1+TR^1+\bar{T}T^1+\bar{T}\bar{T}^1+\bar{T}R^1+RT^1+R\bar{T}^1+RR^1\Big)(f,g).
\end{equation}
We focus on estimates for typical terms such as $TR^1$ and $RR^1$. Other terms can be estimated similarly.

\texttt{Estimate of $TR^1$.} Since
\begin{equation*}
 \begin{split}
 \|S_{q'-1}\D_{k'}^1f\|_{L^\infty}\lesssim \sum_{p\le q'-2}2^p\|\D_{p}\D_{k'}^1f\|_{L^2}\lesssim 2^{q'(1-s)}\|f\|_{\hat{B}^s}.
 \end{split}
\end{equation*}
Since $\mathcal{F}(\D_{q'} f\tilde{\D}_{q'}g)$ is contained in $\beta\{|\xi|\le 2^{q'}\}$ for some $0<\beta$, the inequality above entails
\begin{equation*}
 \begin{split}
\|\D_q\D_k^1(TR^1(f,g))\|_{L^2}&\lesssim \sum_{\substack{|q'-q|\le 3\\k'\ge k-2}}\|S_{q'-1}\D_{k'}^1f\|_{L^\infty}\|\D_{q'}\tilde{\D}_{k'}^1g\|_{L^2}\\
&\lesssim \|f\|_{\hat{B}^s}\sum_{\substack{|q'-q|\le 3\\k'\ge k-2}}2^{q'(1-s)}\|\D_{q'}\tilde{\D}_{k'}^1g\|_{L^2},
 \end{split}
\end{equation*}
and hence combining Holder and convolution inequalities for series gives
$$\|TR^1(f,g)\|_{\hat{B}^{s+t-1}}\lesssim \|f\|_{\hat{B}^s}\|g\|_{\hat{B}^t}.$$

\texttt{Estimate of $RR^1$.}
\begin{equation*}
 \begin{split}
 \|\D_q\D_k^1(RR^1(f,g))\|_{L^2}&\lesssim \sum_{\substack{q'\ge q-2\\k'\ge k-2}}\|\D_{q'}\D_{k'}^1f\|_{L^\infty}\|\tilde{\D}_{q'}\tilde{\D}_{k'}^1g\|_{L^2}\\
&\lesssim 2^q\sum_{\substack{q'\ge q-2\\k'\ge k-2}}\|\D_{q'}\D_{k'}^1f\|_{L^2}\|\tilde{\D}_{q'}\tilde{\D}_{k'}^1g\|_{L^2}.
 \end{split}
\end{equation*}
Thus we have
\begin{equation*}
 \begin{split}
2^{q(s+t-1)}\|\D_q\D_k^1(RR^1(f,g))\|_{L^2}\lesssim \sum_{\substack{q'\ge q-2\\k'\ge k-2}}2^{(q-q')(s+t)}\alpha_{q',k'}\|f\|_{\hat{B}^s}\|g\|_{\hat{B}^t},
 \end{split}
\end{equation*}
and the convolution inequalities for series over $q$ yields
$$\|RR^1(f,g)\|_{\hat{B}^{s+t-1}}\lesssim \|f\|_{\hat{B}^s}\|g\|_{\hat{B}^t}$$
since $s+t>0.$
\end{proof}

Next we turn to the proof of Lemma \ref{cl}. The key idea is to apply the integration by parts to convert the derivative on $f$ or $g$ to the derivative on $\u$.
\begin{proof}[Proof of Lemma \ref{cl}]
To prove \eqref{25b} and \eqref{25}, one can use \eqref{71} to decompose the product $e\cdot\nabla f$ into nine pieces; and then estimate term by term. For illustration, let us consider
\begin{equation}\label{72}
 \begin{split}
&\sum_{\substack{|q'-q|\le 3\\k'\ge k-2}}\Big(G(D)\D_q\D_k^1(S_{q'-1}\D_{k'}^1e_j\D_{q'}\tilde{\D}_{k'}^1\partial_{x_j} f)|G(D)\D_q\D_k^1 f\Big),
 \end{split}
\end{equation}
and the worst term above is 
\begin{equation}\label{73}
\Big(S_{q-1}\D_k^1e_jG(D)\D_{q}\D_{k}^1\partial_{x_j} f|G(D)\D_q\D_k^1 f\Big)
\end{equation}
since the difference between \eqref{72} and \eqref{73} can be estimate with the aid of the first order Taylor's formula. 

For \eqref{73}, integration by parts gives
\begin{equation}\label{75}
 \begin{split}
&\left|\Big(S_{q-1}\D_k^1e_jG(D)\D_{q}\D_{k}^1\partial_{x_j} f|G(D)\D_q\D_k^1 f\Big)\right|\\
&\quad=\left|\Big(S_{q-1}\D_k^1\Dv e G(D)\D_{q}\D_{k}^1 f|G(D)\D_q\D_k^1 f\Big)\right|\\
&\quad\lesssim \|S_{q-1}\D_k^1\Dv e\|_{L^\infty} \|G(D)\D_q\D_k^1 f\|_{L^2}^2\\
&\quad\lesssim
\begin{cases}
\alpha_{q,k}2^{qm+nk}\|e\|_{\hat{B}^2}\|f\|_{\hat{B}^0}\|G(D)\D_q\D_k^1g\|_{L^2}\\
\alpha_{q,k}2^{qm+nk}\min\{2^{-1}, 2^{k-2q}\}\|e\|_{\hat{B}^2}\|f\|_{\tilde{B}^{0,1}}\|G(D)\D_q\D_k^1g\|_{L^2}.
\end{cases}
 \end{split}
\end{equation}

To prove \eqref{26}, we again focus on the piece
\begin{equation*}
 \begin{split}
&\sum_{\substack{|q'-q|\le 3\\k'\ge k-2}}\Big[\Big(G(D)\D_q\D_k^1(S_{q'-1}\D_{k'}^1e_j\D_{q'}\tilde{\D}_{k'}^1\partial_{x_j} f)|\D_q\D_k^1 g\Big)\\
&\quad+\Big(\D_q\D_k^1(S_{q'-1}\D_{k'}^1e_j\D_{q'}\tilde{\D}_{k'}^1\partial_{x_j} g)|G(D)\D_q\D_k^1 f\Big)\Big],
 \end{split}
\end{equation*}
and the worst term above is
\begin{equation*}
 \begin{split}
&\Big(S_{q-1}\D_{k}^1e_jG(D)\D_{q}\D_{k}^1\partial_{x_j} f|\D_q\D_k^1 g\Big)\\
&\quad+\Big(S_{q-1}\D_{k}^1e_j\D_{q}\D_{k}^1\partial_{x_j} g|G(D)\D_q\D_k^1 f\Big),
 \end{split}
\end{equation*}
which equals, using integration by parts
\begin{equation}\label{76}
 \begin{split}
-\Big(S_{q-1}\D_{k}^1\Dv eG(D)\hat{\D}_{q,q_1,q_2} f|\hat{\D}_{q,q_1,q_2} g\Big).
 \end{split}
\end{equation}
Similarly as \eqref{75}, we can estimate \eqref{76} to obtain \eqref{26}.

\end{proof}

Finnaly we turn to the proof of the commutator estimate, Lemma \ref{cl1}.
\begin{proof}[Proof of Lemma \ref{cl1}]
Denote
$$[\u\cdot\nabla, \chi(D)]f=\u\cdot\nabla(\chi(D)f)-\chi(D)(\u\cdot\nabla f).$$
Then estimates \eqref{27a} and \eqref{27b} easily follows from the estimate
\begin{equation}\label{27e}
 \begin{split}
\left|\mathcal{F}\Big([\u\cdot\nabla, \chi(D)]f\Big)(\xi)\right|\lesssim \int_{\R^2}\int_0^1|\eta_i\partial_j\chi(t\xi+(1-t)\eta)||\mathcal{F}(\nabla \u)(\xi-t\eta)||\mathcal{F}(f)(\eta)|dtd\eta.
 \end{split}
\end{equation}
The inequality above follows directly from the the first order Taylor's formula
\begin{equation*}
 \begin{split}
\mathcal{F}\Big([\u\cdot\nabla, \chi(D)]f\Big)(\xi)&=\int_{\R^2}\int_0^1\mathcal{F}(\u_i)(\xi-\eta)\mathcal{F}(f)(\eta)\eta_i\nabla\chi(t\xi+(1-t)\eta)\cdot(\xi-\eta) dt d\eta\\
&=\int_{\R^2}\int_0^1\mathcal{F}(\partial_j\u_i)(\xi-\eta)\mathcal{F}(f)(\eta)\eta_i\partial_j\chi(t\xi+(1-t)\eta) dt d\eta.  
 \end{split}
\end{equation*}
\end{proof}

\bigskip

%%%%%%%%%%%%%%%%%%%%%%%%%%%%%%%%%%%%%%%%%%%%%%%

\end{document}